\documentclass[12pt]{amsart}

\usepackage{amsmath,amssymb,amsthm,mathtools}
\usepackage[a4paper,margin=2.5cm]{geometry}
\usepackage{enumitem,microtype}
\usepackage{xcolor}
\usepackage[hidelinks]{hyperref}
\usepackage{amsthm, color}
\allowdisplaybreaks
\usepackage{citeref}
\usepackage{tikz}
\usepackage{hyperref}

\newtheorem{theorem}{Theorem}[section]
\newtheorem{lemma}[theorem]{Lemma}

\newtheorem{corollary}[theorem]{Corollary}
\theoremstyle{definition}
\newtheorem{definition}[theorem]{Definition}
\newtheorem{remark}[theorem]{Remark}
\numberwithin{equation}{section}

\DeclareMathOperator{\Sec}{Sec}
\DeclareMathOperator{\Ric}{Ric}

\DeclareMathOperator{\Cut}{Cut}
\DeclareMathOperator{\Hess}{Hess}
\DeclareMathOperator{\tr}{tr}
\DeclareMathOperator{\osc}{osc}
\newcommand{\divg}{\operatorname{div}_{M}}
\newcommand{\R}{\mathbb R}
\newcommand{\dd}{\mathrm d}
\newcommand{\cL}{\mathcal L}
\newcommand{\Rm}{\operatorname{Rm}}

\definecolor{sourceblue}{RGB}{25,75,135}

\title[Gradient estimates for mean curvature equations]
{ Interior gradient estimates for mean curvature equations on Riemannian manifolds}

\begin{document}
	\author{Weiyang Li}
\address{School of Mathematical Sciences, University of Science and Technology of China, Hefei, P. R. China}
	\email{liweiyang@ustc.edu.cn}

	\author{Fanheng Xu}
\address{School of Mathematics (Zhuhai), Sun Yat-Sen University, 519082 Zhuhai, People's Republic of China}
	\email{xufh7@mail.sysu.edu.cn}
\thanks{Xu was supported by the National Natural Science Foundation of China (No.12201654)}



\begin{abstract}

In this paper, we study interior gradient estimates for solutions of
\[
 \divg\left(
 \frac{\nabla u}{\sqrt{1+|\nabla u|_g^2}}
 \right)
 =f(x,\nabla u)
\]
 on a complete Riemannian manifold of dimension $n\ge2$. 
 Under a Ricci curvature lower bound and a Hessian comparison
for the distance function, we derive polynomial and exponential
gradient bounds under structural conditions on $f$ and its
horizontal and vertical derivatives. The estimates quantify the dependence on the radius, the
oscillation of the solution, and the geometric and analytic
parameters. In the Euclidean case, we also improve the dependence
on the radius in the gradient estimate of
Xu~\cite[Theorem~1.1(d)]{Xu2026}.
As applications, we obtain Liouville-type results for entire
solutions under suitable curvature and growth assumptions.


\end{abstract}

\maketitle

\section{Introduction}\label{s1}

Let $(M,g)$ be a complete Riemannian manifold of dimension
$n\ge2$, and let $\Omega\subset M$ be open.
We consider
\begin{equation}\label{p1}
 \divg\left(
 \frac{\nabla u}{\sqrt{1+|\nabla u|_g^2}}
 \right)
 =f(x,\nabla u)
 \quad\text{in }\Omega,
\end{equation}
where the gradient and divergence are taken with respect
to $g$, and $f$ is a real-valued function on $TM$.
Our aim is to describe explicitly how the geometry of
$M$ and the dependence of $f$ on the base point and
the tangent vector affect interior gradient estimates.

Set
\begin{equation}\label{def-W}
 W=\sqrt{1+|\nabla u|_g^2}.
\end{equation}
For $u\in C^2(\Omega)$, its graph
\[
 \Sigma=\{(x,u(x)):x\in\Omega\}\subset M\times\mathbb R
\]
has mean curvature $H_\Sigma$ with respect to the upward
unit normal, with the sign convention
\begin{equation}\label{graph-H}
 nH_\Sigma
 =\divg\left(\frac{\nabla u}{W}\right)
 =\frac{1}{W}
  \left(
   \Delta u-\frac{\Hess u(\nabla u,\nabla u)}{W^2}
  \right);
\end{equation}
see \cite[Section~2]{Spruck2007}.
Thus \eqref{p1} prescribes $nH_\Sigma$.

In a local orthonormal frame, write
\[
 u_i=e_i(u),\quad
 u_{i;j}=\Hess u(e_i,e_j),
\]
with summation over repeated indices.
By \eqref{graph-H}, equation \eqref{p1} has the
nondivergence form
\begin{equation}\label{p1-nondiv}
 \cL^{ij}u_{i;j}=Wf(x,\nabla u),
 \quad
 \cL^{ij}=\delta^{ij}-\frac{u_i u_j}{W^2}.
\end{equation}
The geometric notation and curvature convention are
specified in Section~\ref{s21}.

We recall several interior gradient estimates for the prescribed mean
curvature equation obtained from \eqref{p1} by taking $f=H(x)$:
\begin{equation}\label{p-H}
 \divg\left(\frac{\nabla u}{\sqrt{1+|\nabla u|_g^2}}\right)
 =H(x).
\end{equation}
When $H=0$, \eqref{p-H} is the minimal surface equation.

In Euclidean space, interior gradient estimates for the minimal
surface equation were established by Finn~\cite{Finn1954,Finn1963}
in dimension two and by Bombieri, De Giorgi, and
Miranda~\cite{BDM1969} in higher dimensions.
For a $C^2$ solution on $B_R(x_0)\subset\R^n$, $n\ge2$, the
classical estimate takes the form
\begin{equation}\label{eq-Eu1}
 |\nabla u(x_{0})| \leq C_0\exp \left[
 \frac{C_1}{R}\sup\nolimits_{B_R(x_0)}\bigl(u-u(x_0)\bigr)\right],
 \quad C_0,C_1>0,
\end{equation}
where the constants depend only on $n$. Trudinger~\cite{Trudinger1972} gave another proof. For $H\ne0$, Heinz~\cite{Heinz1971}
obtained interior gradient estimates in dimension two under
a Lipschitz assumption on $H$.
Bombieri and Giusti~\cite{BombieriGiusti1973} established
local gradient bounds in higher dimensions under corresponding
regularity assumptions on $H$.

Trudinger~\cite{Trudinger1973} derived exponential gradient
estimates from mean value inequalities on the graph, with
constants depending on bounds for $H$ and its first derivatives.
Wang~\cite{Wang1998} used an auxiliary function to obtain exponential gradient estimates for
nonnegative solutions under bounds on $H$ and $\nabla H$.
Korevaar~\cite{Korevaar1986} also obtained an
exponential estimate using a normal variation argument.

On Riemannian manifolds, curvature enters both the
differentiated equation and the localization argument. For the minimal surface equation,
Rosenberg, Schulze, and Spruck
\cite{RosenbergSchulzeSpruck2013}
assumed
\[
 \Ric_g\ge0,\quad \Sec_g\ge-\mu^2,
 \quad \mu\ge0,
\]
and proved that every nonnegative solution on $B_R(p)$
satisfies
\[
 |\nabla u(p)|_g
 \le C_1\exp\left(
 C_2\frac{u(p)^2}{R^2}
 \left[1+(n-1)\mu R\coth(\mu R)\right]
 \right).
\]
Their argument treats the cut locus and does not
require the ball to lie within the injectivity radius.
More recently, Colombo, Magliaro, Mari, and Rigoli
\cite{ColomboEtAl2022} obtained a global gradient estimate
for positive entire minimal graphs under a Ricci curvature
lower bound alone.
In particular, if
\[
 \Ric_g\ge-(n-1)\kappa^2g
\]
on a complete manifold and $u>0$ is an entire solution
of the minimal surface equation, then
\[
 W\le\exp\bigl(\kappa\sqrt{n-1}\,u\bigr).
\]
For $H\neq 0$, Spruck~\cite{Spruck2007}
proved an exponential interior gradient estimate for
nonnegative solutions on balls whose radius is smaller
than the injectivity radius at the center. Specifically, Spruck proved
\begin{equation*}
 \begin{aligned}
 W(x_{0})\le
 32\max\left\{1,\left(\frac{u(x_{0})}{R}\right)^2\right\}
 \exp\!\left\{16Cu(x_{0})\right\}
 \exp\!\left\{16C\left(\frac{u(x_{0})}{R}\right)^2\right\}.
 \end{aligned}
\end{equation*}
The constant $C$ depends on the $C^1$ bounds for $H$,
a sectional curvature lower bound, and control of
the Laplacian of the squared distance.

For mean curvature equations whose right-hand side depends on $u$, consider
\begin{equation}\label{p-fu}
 \divg\left(\frac{\nabla u}{\sqrt{1+|\nabla u|_g^2}}\right)
 =f(x,u).
\end{equation}
In Euclidean domains, interior gradient bounds can
be obtained for bounded $C^3$ solutions of \eqref{p-fu}
under appropriate assumptions on $f$, including
$f_u\ge0$; see
\cite[Section~16.2]{GilbargTrudinger2001}.
Ma and Xu~\cite{MaXu2016} obtained gradient bounds for bounded
solutions of \eqref{p-fu} with Neumann boundary conditions, under
corresponding assumptions on $f$ and regularity bounds
on the boundary data.
Xu and Zhang~\cite{XuZhang2017} extended these estimates
to Riemannian domains, with constants also depending
on the domain geometry and bounds for the Ricci curvature.


In this paper, we extend Xu's work~\cite{Xu2026}
to complete Riemannian manifolds.
In \cite{Xu2026}, the author studied \eqref{p1}
in Euclidean space with $f=f(\nabla u)$ and obtained
polynomial and exponential estimates under
inequalities relating $f$ to its derivative.
Here we consider \eqref{p1} with $f=f(x,\nabla u)$
on Riemannian manifolds. We distinguish between a Ricci curvature lower bound,
which enters the Bochner identity, and a Hessian
comparison for the distance function, which controls
the cutoff terms.
The dependence on $x$ is measured by the horizontal
derivative of $f$, while its dependence on the tangent
vector is measured by the vertical derivative of $f$.

Let $d_g$ denote the Riemannian distance.
For $p\in M$, write
\[
 r(x)=d_g(p,x),\quad
 B_R(p)=\{x\in M:r(x)<R\}.
\]
Let $\Cut(p)$ denote the cut locus of $p$.
The function $r$ is smooth outside
$\{p\}\cup\Cut(p)$, and $r^2$ is smooth near $p$;
see \cite{Petersen2016}.

For $\mu\ge0$ and $s>0$, set
\begin{equation}\label{eq-comp1}
 c_\mu(s)=
 \begin{cases}
  \mu\coth(\mu s),&\mu>0,\\
  \frac{1}{s},&\mu=0.
 \end{cases}
\end{equation}
We formulate the distance Hessian comparison using
smooth upper supports (see \cite{Calabi1958Duke,MantegazzaEtAl2014,Petersen2016}) at the cut locus.
\begin{definition}\label{def-compare}
The distance function $r=d_g(p,\cdot)$ satisfies the
Hessian comparison on $B_R(p)$ with parameter
$\mu\ge0$ if the following conditions hold.

\begin{enumerate}
\item[(i)]
On $B_R(p)\setminus(\{p\}\cup\Cut(p))$,
\begin{equation}\label{eq-comp2}
 \Hess r\le c_\mu(r)(g-\dd r\otimes\dd r).
\end{equation}

\item[(ii)]
For every $q\in B_R(p)\cap\Cut(p)$ and every
sufficiently small $\varepsilon>0$, there is a smooth
function $r_{q,\varepsilon}$ on a neighborhood of $q$
contained in $B_R(p)$ such that
\begin{equation}\label{eq-supp}
 r_{q,\varepsilon}\ge r,\quad
 r_{q,\varepsilon}(q)=r(q),
\end{equation}
and, at $q$,
\begin{equation}\label{eq-comp3}
 \Hess r_{q,\varepsilon}
 \le
 \bigl(c_\mu(r(q))+\varepsilon\bigr)
 \bigl(g-\dd r_{q,\varepsilon}
          \otimes\dd r_{q,\varepsilon}\bigr)
 +\varepsilon g.
\end{equation}
\end{enumerate}
All tensor inequalities are understood as inequalities
of quadratic forms.
\end{definition}
\begin{remark}\label{rem1}
A sectional curvature lower bound
$\Sec_g\ge-\mu^2$ implies
Definition~\ref{def-compare} by Hessian comparison
and the usual displacement of the base point along
a minimizing geodesic; see \cite{Petersen2016}
and the proof of Corollary~\ref{cor1}.
Moreover, the upper-support property
\eqref{eq-supp} implies
\[
 |\nabla r_{q,\varepsilon}(q)|_g\le1,\quad
 g-\dd r_{q,\varepsilon}\otimes
       \dd r_{q,\varepsilon}\ge0
 \quad\text{at }q,
\]
as shown in Lemma~\ref{lem:cutoff}.
\end{remark}

We next recall the horizontal and vertical derivatives.
Their coordinate expressions and the corresponding
chain rule are given in Section~\ref{s22}.

\begin{definition}\label{def-f-deriv}
Let
\[
 0_M=\{(x,0_x):x\in M\}
\]
be the zero section, and let
$f\in C^1(TM\setminus0_M)$.
For $(x,\xi)\in TM\setminus0_M$, the vectors
\[
 \nabla^h f(x,\xi),\quad
 \nabla^v f(x,\xi)\in T_xM
\]
are defined by
\[
 \begin{aligned}
 \langle\nabla^h f(x,\xi),X\rangle_g
 &=
 \left.\frac{\dd}{\dd t}\right|_{t=0}
 f(\gamma(t),P_{0,t}\xi),\\
 \langle\nabla^v f(x,\xi),Y\rangle_g
 &=
 \left.\frac{\dd}{\dd t}\right|_{t=0}
 f(x,\xi+tY),
 \end{aligned}
\]
for $X,Y\in T_xM$.
Here $\gamma(0)=x$, $\dot\gamma(0)=X$, and
$P_{0,t}$ is parallel transport along $\gamma$
with respect to the Levi--Civita connection.
\end{definition}
In particular, wherever $\nabla u\ne0$,\[
X[f(\cdot,\nabla u)]
=
\langle\nabla^h f(x,\nabla u(x)),X\rangle_g
+
\langle\nabla^v f(x,\nabla u(x)),
        \nabla_X\nabla u\rangle_g.
\]
This identity separates the two types of derivatives
that enter the differentiated equation.

\begin{theorem}\label{thm-main}\rm
Let $(M,g)$ be a complete Riemannian manifold of
dimension $n\ge2$, and let $B_R(p)\subset\Omega$.
Assume that, for some $\kappa,\mu\ge0$,
\begin{equation}\label{eq-geometry}
 \Ric_g\ge-(n-1)\kappa^2g
 \quad\text{on }B_R(p),
\end{equation}
and that $r=d_g(p,\cdot)$ satisfies
Definition~\ref{def-compare} with parameter $\mu$.

Suppose
$f\in C^0(TM)\cap C^1(TM\setminus0_M)$ and
\begin{equation}\label{eq-f}
 \langle\nabla^h f(x,\xi),\xi\rangle_g
 \ge-\sigma_h|\xi|_g|f(x,\xi)|
 \quad\text{on }TB_R(p)\setminus0_M,
\end{equation}
where $\sigma_h\ge0$.

Let $u\in C^1(B_R(p))\cap
 C^3\bigl(B_R(p)\cap\{|\nabla u|_g>0\}\bigr)$
 be a weak solution of \eqref{p1} and the equation hold classically wherever $|\nabla u|_g>0$. Set
\begin{equation}\label{eq-scales}
 \begin{aligned}
 &L=\osc_{B_R(p)}u=\sup\nolimits_{B_R(p)}u-\inf\nolimits_{B_R(p)}u,\\
 &K(R,\mu)=\frac{1+\mu R}{R^2},\\
 &\Lambda(R,\mu,\kappa,\sigma_h)
 =K(R,\mu)+(n-1)\kappa^2+n\sigma_h^2.
 \end{aligned}
\end{equation}
All structural inequalities below are imposed on
$TB_R(p)\setminus0_M$, with $f$ and its derivatives
evaluated at $(x,\xi)$.

\begin{enumerate}
\item[(i)]
If $\theta>0$, $m_1,m_2>0$, and
\begin{equation}\label{A1}
 f^2-m_1|\xi|_g^2|\nabla^v f|_g^2
 \ge m_2|\xi|_g^{2\theta},
\end{equation}
then
\begin{equation}\label{eq:a-grad}
 |\nabla u(p)|_g
 \le
 C\bigl[\Lambda(R,\mu,\kappa,\sigma_h)\bigr]^{
       \frac{1}{2\theta}},
 \quad C=C(n,m_1,m_2,\theta).
\end{equation}

\item[(ii)]
If $\theta>1$, $m_1,m_2,m_3>0$, and
\begin{equation}\label{A2}
 \begin{cases}
 f^2-m_1(1+|\xi|_g^2)\log(1+|\xi|_g^2)
       |\nabla^v f|_g^2
 \ge m_2\log^{2\theta}(1+|\xi|_g^2),\\
 |f|\le m_3\log^\theta(1+|\xi|_g^2),
 \end{cases}
\end{equation}
then
\begin{equation}\label{eq:b-grad}
 \begin{aligned}
 |\nabla u(p)|_g^2
 \le C\max\Big\{
 &\exp\Big(
 C\bigl[\Lambda(R,\mu,\kappa,\sigma_h)\bigr]^{
          \frac{1}{2(\theta-1)}}
 \Big)-1,\ \ \bigl[\Lambda(R,\mu,\kappa,\sigma_h)\bigr]^{
          \frac{1}{2\theta-1}}
 \Big\},
 \end{aligned}
\end{equation}
where $C=C(n,m_1,m_2,m_3,\theta)$.

\item[(iii)]
If $0<\theta\le1$, $m_1>0$, and
\begin{equation}\label{A3}
 |f(x,\xi)|=m_1\log^\theta(1+|\xi|_g^2),
\end{equation}
then $\nabla^h f=0$.
For every $0<\beta<1$,
\begin{equation}\label{eq:c-grad}
 \begin{aligned}
 |\nabla u(p)|_g
 \le\max\Big\{
 &C_1\bigl[(L+1)\Lambda(R,\mu,\kappa,0)\bigr]^{
          \frac{1}{2\theta}},\ \ C_2\bigl[(L+1)\Lambda(R,\mu,\kappa,0)\bigr]^{
          \frac{1}{1-\beta}}
 \Big\},
 \end{aligned}
\end{equation}
where $C_1,C_2$ depend only on $n,m_1,\theta,\beta$.

\item[(iv)]
If $m_1,m_2>0$ and
\begin{equation}\label{A4}
 \begin{cases}
 f^2-m_1(1+|\xi|_g^2)^2\log(1+|\xi|_g^2)
       |\nabla^v f|_g^2\ge0,\\
 |f|\le m_2,
 \end{cases}
\end{equation}
then
\begin{equation}\label{eq:d-grad}
	\begin{aligned}
 |\nabla u(p)|_g^2
 \le\max\Big\{
 &\exp\!\bigl(C(L+1)^2\Lambda(R,\mu,\kappa,\sigma_h)\bigr)-1,\\
 &\quad \exp\!\bigl(C(L+1)^2\left[\Lambda(R,\mu,\kappa,\sigma_h)\right]^{\frac{1}{2}}\bigr)-1
 \Big\},
	\end{aligned}
\end{equation}
where $C=C(n,m_1,m_2)$.

\item[(v)]
If $f\equiv0$, then
\begin{equation}\label{eq:e-grad}
 |\nabla u(p)|_g^2
 \le \exp\!\bigl(CL^2\Lambda(R,\mu,\kappa,0)\bigr)-1
\end{equation}
where $C=C(n)$.
\end{enumerate}
\end{theorem}

A sectional curvature lower bound gives both
geometric assumptions in Theorem~\ref{thm-main}.

\begin{corollary}\label{cor1}\rm
In Theorem~\ref{thm-main}, replace the Ricci curvature
and distance Hessian assumptions by
\[
 \Sec_g\ge-\rho^2
 \quad\text{on }B_R(p),
 \quad \rho\ge0.
\]
Then all conclusions hold with
\[
 \kappa=\mu=\rho,\quad
 \Lambda(R,\rho,\rho,\sigma_h)
 =\frac{1+\rho R}{R^2}
  +(n-1)\rho^2+n\sigma_h^2.
\]
In particular, on $\mathbb R^n$, if $\sigma_h=0$
is admissible in \eqref{eq-f}, then
\[
 \Lambda(R,0,0,0)=\frac{1}{R^2}.
\]
\end{corollary}

For entire solutions, letting the radius tend to
infinity gives the following Liouville-type result.

\begin{corollary}\label{cor2}\rm
Let $(M,g)$ be a complete, connected, noncompact
Riemannian manifold of dimension $n\ge2$ satisfying
\[
 \Ric_g\ge0,\quad \Sec_g\ge-\rho^2
\]
for some fixed $\rho\ge0$.
Suppose
\[
 f\in C^0(TM)\cap C^1(TM\setminus0_M),\quad
 \langle\nabla^h f(x,\xi),\xi\rangle_g\ge0
 \quad\text{on }TM\setminus0_M,
\]
and let $u\in C^3(M)$ solve \eqref{p1} on $M$.

Fix $p\in M$ and write $L=\osc_{B_R(p)}u$.
In each case below, the structural condition holds
on $TM\setminus0_M$, with the parameter ranges of
Theorem~\ref{thm-main} and constants independent
of $x$ and $\xi$.
Then $u$ is constant under any of the following
assumptions:
\begin{enumerate}
\item[(i)]
Condition \eqref{A1} holds.

\item[(ii)]
Condition \eqref{A2} holds.

\item[(iii)]
Condition \eqref{A3} holds and, as $R\to\infty$,
\[
 L=
 \begin{cases}
 o(R),&\rho>0,\\
 o(R^2),&\rho=0.
 \end{cases}
\]

\item[(iv)]
Condition \eqref{A4} holds and, as $R\to\infty$,
\[
 L=
 \begin{cases}
 o\bigl(R^{\frac{1}{4}}\bigr),&\rho>0,\\
 o\bigl(R^{\frac{1}{2}}\bigr),&\rho=0.
 \end{cases}
\]

\item[(v)]
We have $f\equiv0$ and, as $R\to\infty$,
\[
 L=
 \begin{cases}
 o\bigl(R^{\frac{1}{2}}\bigr),&\rho>0,\\
 o(R),&\rho=0.
 \end{cases}
\]
\end{enumerate}
\end{corollary}

\begin{remark}\label{rem2}
	We explain the geometric and analytic assumptions
in Theorem~\ref{thm-main}.
\begin{itemize}
\item[(1)] 
A Ricci curvature lower bound gives the Laplacian
comparison
\[
 \Ric_g\ge-(n-1)\kappa^2g
 \quad\Longrightarrow\quad
 \Delta r\le(n-1)c_\kappa(r)
\]
away from the base point and its cut locus;
see \cite{Petersen2016}.
For $n\ge3$, it does not in general give the
tensor inequality \eqref{eq-comp2}.
The distinction matters because the cutoff calculation
contains
\[
 \cL^{ij}r_{i;j}
 =
 \Delta r
 -\frac{\Hess r(\nabla u,\nabla u)}{W^2},
\]
rather than $\Delta r$ alone.
The distance Hessian assumption is used to control
this contraction in the present localization argument.


\item[(2)] Condition \eqref{eq-f} only bounds the component of
$\nabla^h f$ in the direction $\xi$.
In particular,
\[
 \langle\nabla^h f(x,\xi),\xi\rangle_g\ge0
 \quad\Longrightarrow\quad
 \sigma_h=0\ \text{is admissible},
\]
even when $\nabla^h f\ne0$.

For example,  on $B_R(0)\subset\R^n$, take
\[
 f(x,\xi)=c|\xi|+\lambda x\cdot\xi,
 \quad \lambda>0,\quad c>\lambda R.
\]
Then
\[
\begin{aligned}
 \nabla^hf=\lambda\xi,\quad
 \langle\nabla^hf,\xi\rangle=\lambda|\xi|^2\ge0,
\end{aligned}
\]
and
\[
\begin{aligned}
 f^2-m_1|\xi|^2|\nabla^vf|^2
 \ge\bigl[(c-\lambda R)^2-m_1(c+\lambda R)^2\bigr]|\xi|^2.
\end{aligned}
\]
Thus \eqref{A1} holds with $\theta=1$ and
\[
 m_2=(c-\lambda R)^2-m_1(c+\lambda R)^2>0
\]
whenever
\[
 0<m_1<
 \left(\frac{c-\lambda R}{c+\lambda R}\right)^2.
\]
Moreover, $\sigma_h=0$ is admissible although
$\nabla^h f\ne0$.


\item [(3)] 
We give examples satisfying the structural conditions of Theorem~\ref{thm-main}.
Suppose that
\[
 \Phi\in C([0,\infty))\cap C^1((0,\infty)),
 \quad f(x,\xi)=\Phi(|\xi|_g).
\]
Since parallel transport preserves the norm,
\[
 \nabla^h f=0,\quad
 \nabla^v f
 =\Phi'(|\xi|_g)\frac{\xi}{|\xi|_g},\quad
 |\nabla^v f|_g=|\Phi'(|\xi|_g)|.
\]
Thus the structural conditions reduce to inequalities in one variable.
\begin{itemize}
\item [$\bullet$]For $\Phi(t)=ct^\theta$, $c,\theta>0$,
\[
 \Phi(t)^2-m_1t^2|\Phi'(t)|^2
 =c^2(1-m_1\theta^2)t^{2\theta}.
\]
Condition \eqref{A1} therefore holds with
\[
 0<m_1<\frac{1}{\theta^2},\quad
 m_2=c^2(1-m_1\theta^2).
\]

\item[$\bullet$] 
For $\Phi(t)=c\sqrt{1+t^2}$, with $c>0$ and $0<m_1<1$,
condition \eqref{A1} holds with $\theta=1$ and
$m_2=c^2(1-m_1)$.
This form also arises, after rescaling the dependent variable,
in the elliptic regularization of inverse mean curvature flow;
see \cite{HuiskenIlmanen2001}.

\item[$\bullet$] For $\Phi(t)=c\log^\theta(1+t^2)$, $c,\theta>0$,
\[
 \frac{(1+t^2)\log(1+t^2)|\Phi'(t)|^2}
      {\Phi(t)^2}
 =
 \frac{4\theta^2t^2}
      {(1+t^2)\log(1+t^2)}
 \le4\theta^2.
\]
If $\theta>1$, condition \eqref{A2} holds with
\[
 0<m_1<\frac{1}{4\theta^2},\quad
 m_2=c^2(1-4m_1\theta^2),\quad m_3=c.
\]
If $0<\theta\le1$, condition \eqref{A3} holds
with $m_1=c$.

\item[$\bullet$]For $
 \Phi(t)=\frac{ct}{\sqrt{1+t^2}}$, $c>0$,
we have
\[
 |\Phi'(t)|=\frac{c}{(1+t^2)^{\frac{3}{2}}},
 \quad |\Phi(t)|\le c,
\]
and
\[
 \begin{aligned}
 &\Phi(t)^2
 -m_1(1+t^2)^2\log(1+t^2)|\Phi'(t)|^2\\
 &\quad=
 \frac{c^2}{1+t^2}
 \bigl[t^2-m_1\log(1+t^2)\bigr]
 \ge0
 \end{aligned}
\]
whenever $0<m_1\le1$.
Thus \eqref{A4} holds with $m_2=c$.

\end{itemize}
\end{itemize}
\end{remark}

\begin{remark}\label{rem3}
In the Euclidean case $M=\mathbb R^n$ and $f=f(\xi)$,
part \textup{(iv)} of Theorem~\ref{thm-main} gives,
under \eqref{A4},
\[
 \log(1+|\nabla u(0)|^2)
 \le C(L+1)^2\max\left\{\frac{1}{R^2},\frac{1}{R}\right\}.
\]
For comparison, \cite[Theorem~1.1(d)]{Xu2026} gives
\[
 \log(1+|\nabla u(0)|^2)
 \le C(L+1)^2
 \max\left\{
 \frac{1}{R^2},\frac{1}{R^{\frac{2}{3}}}
 \right\}.
\]
Consequently, our estimate yields constancy of entire
solutions under
\[
 L=o(R^{\frac{1}{2}}),
\]
whereas the latter estimate yields constancy under
$L=o(R^{\frac{1}{3}})$.
Thus the improvement concerns the admissible growth
of the oscillation as the radius tends to infinity.
Under the geometric assumptions of
Corollary~\ref{cor2} with $\rho>0$, the corresponding
condition is $L=o(R^{\frac{1}{4}})$.
\end{remark}

\begin{remark}\label{rem4}
The distance Hessian comparison may be sharper than
the bound obtained from a uniform sectional curvature
lower bound.

Consider
\[
 (M,g)=
 \left(
 \mathbb R^n,
 \dd r^2+\psi(r)^2g_{\mathbb S^{n-1}}
 \right),
 \quad
 \psi(r)=\frac{r}{1+r^2},
\]
with $p$ the origin.
The metric extends smoothly across $p$ and is complete.
The point $p$ is a pole, and the warped product formulas
give
\[
 \Hess r
 =
 \frac{\psi'}{\psi}
 (g-\dd r\otimes\dd r)
 =
 \frac{1-r^2}{r(1+r^2)}
 (g-\dd r\otimes\dd r)
 \le
 \frac{1}{r}(g-\dd r\otimes\dd r);
\]
see \cite{Petersen2016}.
Thus Definition~\ref{def-compare} holds with $\mu=0$.

On the other hand,
\[
 K_{\mathrm{rad}}
 =
 -\frac{\psi''}{\psi}
 =
 -\frac{1}{8}
 +\frac{(r^2-7)^2}{8(1+r^2)^2}
 \ge-\frac{1}{8}.
\]
For tangential two-planes, when $n\ge3$,
\[
 K_{\mathrm{tan}}
 =
 \frac{1-(\psi')^2}{\psi^2}
 =
 \frac{6+5r^2+4r^4+r^6}{(1+r^2)^2}>0.
\]
Hence the optimal uniform sectional curvature lower
bound is $-\frac{1}{8}$.
Using this bound alone would give the comparison
parameter $\mu=\frac{1}{2\sqrt{2}}$, whereas the direct
calculation allows $\mu=0$.

Theorem~\ref{thm-main} therefore permits the choices
\[
 \mu=0,\quad \kappa=\frac{1}{2\sqrt{2}},
 \quad K(R,0)=\frac{1}{R^2}.
\]
This improves the distance contribution while
retaining the Ricci curvature contribution.
\end{remark}

\begin{remark}\label{rem5}
	We make three observations concerning Corollary~\ref{cor2}.
\begin{enumerate}
\item[(1)]
In parts \textup{(iii)--(v)} of
Corollary~\ref{cor2}, replacing the little-$o$
conditions by the corresponding big-$O$ conditions
gives
\[
 \sup\nolimits_M|\nabla u|_g<\infty.
\]
Indeed, applying the estimates on $B_{R/2}(q)$
and using $B_{R/2}(q)\subset B_R(p)$ for large $R$
gives a bound independent of $q$.
In particular, an entire minimal graph over
$\mathbb R^n$ satisfying
\[
 \osc_{B_R(0)}u=O(R)
\]
has bounded gradient and is affine by
\cite{Moser1961}.
Without a growth assumption, all entire minimal
graphs are affine for $n\le7$;
see
\cite{Almgren1966,Bernstein1927,DeGiorgi1965,Fleming1962,
Simons1968}.
Nonaffine examples exist for $n\ge8$;
see \cite{BombieriDeGiorgiGiusti1969}.

\item[(2)]
The minimal case of Corollary~\ref{cor2} is a
consequence of the present local estimates.
Stronger global results are available under
$\Ric_g\ge0$ alone:
positive entire solutions are constant by
\cite{ColomboEtAl2022}, and sublinear growth of
the negative part suffices by \cite{Ding2025}.
Thus the sectional curvature assumption in our
corollary belongs to the distance-cutoff argument
used here.


\item[(3)]Any constant solution satisfies $f(x,0_x)=0$ for every $x\in M$.
Consequently, under the remaining assumptions of
Corollary~\ref{cor2}, no entire solution exists if
$f(x,0_x)\ne0$ at some point.
\end{enumerate}
\end{remark}

The paper is organized as follows.
Section~\ref{s2} develops the horizontal and vertical
derivative formulas, the elliptic properties of
$(\cL^{ij})$, the Bochner inequality, and the
distance cutoff estimates.
Section~\ref{s3} proves Theorem~\ref{thm-main}.
Section~\ref{s4} proves
Corollaries~\ref{cor1} and~\ref{cor2}.

\section{Preliminaries}\label{s2}

\subsection{Geometric notation and conventions}\label{s21}
Let $\nabla$ denote the Levi--Civita connection of $g$.
For a $C^{2}$ function $u$, we write
\[
\begin{aligned}
&\dd u(X)=X(u)=\langle\nabla u,X\rangle_g,\\
&\Hess u(X,Y)=(\nabla_X\dd u)(Y)
             =\langle\nabla_X\nabla u,Y\rangle_g,\\
&\Delta u=\tr_g\Hess u.
\end{aligned}
\]

Our curvature convention is
\[
\begin{aligned}
&\Rm(X,Y)Z
=\nabla_X\nabla_YZ-\nabla_Y\nabla_XZ-\nabla_{[X,Y]}Z,\\
&\mathcal R(X,Y,Z,T)
=\langle\Rm(X,Y)Z,T\rangle_g.
\end{aligned}
\]
For linearly independent $X,Y\in T_xM$, write
$[X\wedge Y]=\operatorname{span}\{X,Y\}$. Then
\[
\Sec([X\wedge Y])
=\frac{\mathcal R(X,Y,Y,X)}
{|X|_g^2|Y|_g^2-\langle X,Y\rangle_g^2}.
\]

In a local orthonormal frame $\{e_i\}_{i=1}^n$,
\[
\Ric(X,Y)=\sum_{i=1}^n\mathcal R(X,e_i,e_i,Y).
\]
For $u\in C^3$, we use the component notation
\begin{equation*}
\begin{aligned}
&u_i=\dd u(e_i),\\
&u_{i;j}=(\nabla_{e_j}\dd u)(e_i)=\Hess u(e_i,e_j),\\
&u_{i;jk}=(\nabla_{e_k}\Hess u)(e_i,e_j).
\end{aligned}
\end{equation*}
Repeated component indices are summed. In particular,
\[
\nabla u=u_i e_i,\quad
u_{i;j}=u_{j;i},\quad
\Delta u=u_{i;i}.
\]

For pointwise calculations at $q$, we choose the frame so that
\[\left\langle e_{i},e_{j}\right\rangle_{g}(q)=\delta_{ij},\quad
\nabla_{e_i}e_j(q)=0.
\]
Then
\[
u_{i;j}(q)=e_j(u_i)(q),\quad
u_{i;jk}(q)=e_k(u_{i;j})(q).
\]

\subsection{Horizontal and vertical derivatives on \texorpdfstring{$TM$}{TM}}\label{s22}

We verify the derivatives in Definition~\ref{def-f-deriv}
and then establish their chain rule.

The horizontal derivative of $f$ varies the base point while
parallel transporting the tangent vector; the vertical
derivative of $f$ varies the vector within a fixed tangent space. These constructions correspond to the standard horizontal
and vertical lifts; see
\cite{Fourtzisetal2023AGAG,KolarMichorSlovak1993-book,Sasaki1958}.
We verify the formulas in local coordinates.
Let $f\in C^1(TM\setminus0_M)$ and fix
$(x,\xi)\in TM\setminus0_M$.
For $X\in T_xM$, choose a smooth curve $\gamma$ satisfying
$\gamma(0)=x$ and $\dot\gamma(0)=X$.
Parallel transport of $\xi$ along $\gamma$ defines a smooth curve in $TM$:
\[
 \widetilde\gamma(t)
 :=
 \left(\gamma(t),V(t)\right),
 \quad V(t):=P_{0,t}\xi.
\]
The horizontal lift of $X$ at $(x,\xi)$ is defined by \[X^h=\dot{\widetilde\gamma}(0).\] 
Choose local coordinates $(x^i)$ near $x$, and write
\[
 X=X^i\partial_i|_x,
 \quad
 V(t)=V^k(t)\partial_k|_{\gamma(t)}.
\]
Let $\Gamma^k_{ij}$ denote the Christoffel symbols of
$\nabla$ in these coordinates. 
In the induced coordinates $(x^i,\xi^k)$ on $TM$, the curve
$\widetilde\gamma$ is represented by
$(\gamma^i(t),V^k(t))$, where $\gamma^i=x^i\circ\gamma$.
Since $V$ is parallel along $\gamma$, its components satisfy
\[
 \dot V^k(t)
 +\Gamma^k_{ij}(\gamma(t))\dot\gamma^i(t)V^j(t)=0,
 \quad V^k(0)=\xi^k.
\]

Thus 
\begin{equation*}
 \begin{aligned}
X^{h}&=\dot{\widetilde\gamma}(0)\\
 &=
 X^i\left.\frac{\partial}{\partial x^i}\right|_{(x,\xi)}
 +
 \dot V^k(0)
 \left.\frac{\partial}{\partial\xi^k}\right|_{(x,\xi)}\\
 &=
 X^i\left(
 \left.\frac{\partial}{\partial x^i}\right|_{(x,\xi)}
 -
 \Gamma^k_{ij}(x)\xi^j
 \left.\frac{\partial}{\partial\xi^k}\right|_{(x,\xi)}
 \right).
 \end{aligned}
\end{equation*}
Therefore, $X^{h}$ is determined by $(x,\xi,X)$,
independently of the choice of $\gamma$, and depends linearly
on $X$.
Moreover, preservation of the norm gives
\[
 |V(t)|_{g_{\gamma(t)}}=|\xi|_{g_x}>0,
\]
so $\widetilde\gamma$ lies in the domain $TM\setminus0_M$ of $f$. Then \begin{equation}\label{eq-hor-deriv}
 \begin{aligned}
 \left.df \right|_{(x,\xi)}(X^{h})=\left.\frac{\dd}{\dd t}\right|_{t=0}
 f\circ \widetilde{\gamma}=\left.\frac{\dd}{\dd t}\right|_{t=0}
 f(\gamma(t),P_{0,t}\xi)=
 X^i\left(
 \frac{\partial f}{\partial x^i}
 -
 \Gamma^k_{ij}\xi^j\frac{\partial f}{\partial\xi^k}
 \right).
 \end{aligned}
\end{equation}

The map $X\mapsto \left.df \right|_{(x,\xi)}(X^{h})$ is therefore a
well-defined linear functional on $T_xM$. By the Riesz representation theorem, there is a unique vector
$\nabla^hf(x,\xi)\in T_xM$ satisfying
\[
 \langle\nabla^hf(x,\xi),X\rangle_g
 =\left.df \right|_{(x,\xi)}(X^{h}),\quad X\in T_xM.
\]
This proves the well-definedness of the horizontal derivative
in Definition~\ref{def-f-deriv}.

For the vertical derivative, we keep the base point $x$ fixed
and vary $\xi$ within the fiber $T_xM$. We will use the canonical vertical lift associated with the
vector bundle structure of $TM$; see \cite{GraciaPons2001,KolarMichorSlovak1993-book}. 
Given $Z\in T_xM$, consider the curve
\[
 \beta(t):=(x,\xi+tZ).
\]
The vertical lift of $Z$ at $(x,\xi)$ is defined by 
\[Z^{v}=\dot{\beta}(0).\]
In the induced bundle coordinates $(x^{i},\xi^{k})$, writing $Z=Z^k\partial_k|_x$, 
\[
 \beta(t)=\bigl(x^i,\xi^k+tZ^k\bigr).
\]
 Hence
\begin{equation*}
 Z^v:=
 \dot\beta(0)
 =
 Z^k\left.\frac{\partial}{\partial\xi^k}\right|_{(x,\xi)}
 \in T_{(x,\xi)}TM.
\end{equation*}
Therefore, the map
\[
 T_xM\longrightarrow\operatorname{span}\left\{
 \left.\frac{\partial}{\partial\xi^1}\right|_{(x,\xi)},
 \ldots,
 \left.\frac{\partial}{\partial\xi^n}\right|_{(x,\xi)}
 \right\},
 \quad Z\longmapsto Z^{v},
\]
is a linear isomorphism. For sufficiently small $|t|$, the curve $\beta$ lies in $TM\setminus0_M$. Hence
\begin{equation}\label{eq-ver-deriv}
 \begin{aligned}
 df|_{(x,\xi)}(Z^{v})=\left.\frac{\dd}{\dd t}\right|_{t=0}f\circ \beta=\left.\frac{\dd}{\dd t}\right|_{t=0}f(x,\xi+tZ)=\dd f|_{(x,\xi)}\bigl(\dot\beta(0)\bigr)=Z^k\frac{\partial f}{\partial\xi^k}.
 \end{aligned}
\end{equation}
The map $Z\mapsto\dd f|_{(x,\xi)}(Z^v)$ is therefore
a linear functional on $T_xM$. By Riesz representation again, there is a unique vector
$\nabla^vf(x,\xi)\in T_xM$ satisfying
\[
 \langle\nabla^vf(x,\xi),Z\rangle_g
 =
 \dd f|_{(x,\xi)}(Z^{v}),
 \quad Z\in T_xM.
\]
 This proves the well-definedness of the vertical derivative
 in Definition~\ref{def-f-deriv}.

 \begin{lemma}\label{lem:chain}\rm
Let $f\in C^1(TM\setminus0_M)$ and $Y$ be a $C^1$
vector field on an open set $U\subset M$.
Then $\widetilde f(y)=f(y,Y(y))$ is $C^1$ on
$\{y\in U:Y(y)\ne0\}$, and, for every 
$x\in U$ with
$Y(x)\ne0$ and every $X\in T_xM$,
\[
X\widetilde f
=
\langle\nabla^h f(x,Y(x)),X\rangle_g
+
\langle\nabla^v f(x,Y(x)),\nabla_XY\rangle_g.
\]
In particular, if $u\in C^2(U)$, then for every $x\in U$
with $\nabla u(x)\ne0$ and every $X\in T_xM$,
\[
X[f(\cdot,\nabla u)]
=
\langle\nabla^h f(x,\nabla u(x)),X\rangle_g
+
\langle\nabla^v f(x,\nabla u(x)),
        \nabla_X\nabla u\rangle_g.
\]
\end{lemma}

\begin{proof}

Since $Y(x)\ne0$, the vector field $Y$ does not vanish
on a neighborhood of $x$. The composition $\widetilde f$
is therefore $C^1$ on this neighborhood.


In local coordinates $(x^{i},\xi^{k})$, write
$Y=Y^j\frac{\partial}{\partial x^j}$, $X=X^i\left.\frac{\partial}{\partial x^i}\right|_x$, and 
\[(y,Y(y))=(y^{1},\cdots,y^{n},Y^{1}(y),\cdots,Y^{n}(y)).\]
Recall that
\[
\nabla_XY
=
X^i\left(\frac{\partial Y^{k}}{\partial x^{i}}+\Gamma^k_{ij}Y^j\right)
\frac{\partial }{\partial x^{k}}.
\]
Now, we choose a smooth curve $\gamma:(-\varepsilon,\varepsilon)\rightarrow U$ such that $\gamma(0)=x$, $\dot{\gamma}(0)=X$ and use \eqref{eq-hor-deriv} and \eqref{eq-ver-deriv} to obtain
\[
\begin{aligned}
X\widetilde f
&=\left.\frac{d}{dt}\right|_{t=0}\widetilde f (\gamma(t))=\left.\frac{d}{dt}\right|_{t=0}f (\gamma(t),Y(\gamma(t)))\\
&=\left.\left[\frac{\partial f}{\partial x^{i}}\frac{\partial x^{i}(\gamma(t))}{\partial t}+\frac{\partial f}{\partial \xi^{k}}\frac{\partial Y^{k}}{\partial x^{j}}\frac{\partial x^{j}(\gamma(t))}{\partial t}\right]\right|_{t=0}\\
&=X^{i}\left(\frac{\partial f}{\partial x^{i}}-\Gamma^{k}_{ij}Y^{j}(x)\frac{\partial f}{\partial \xi^{k}}\right)+\left(X^{j}\frac{\partial Y^{k}(x)}{\partial x^{j}}+\Gamma^{k}_{ij}X^{i}Y^{j}(x)\right)\frac{\partial f}{\partial \xi^{k}}\\
&=df(X^{h})+df((\nabla_{X}Y)^{v})\\
&=
\langle\nabla^h f(x,Y(x)),X\rangle_g
+
\langle\nabla^v f(x,Y(x)),\nabla_XY\rangle_g.
\end{aligned}
\]
Taking $Y=\nabla u$ proves the last assertion.

\end{proof}
 
\subsection{The principal second-order operator}\label{s23}
Let $u\in C^1(\Omega)$, with $W$ and $\cL^{ij}$ as in
\eqref{def-W} and \eqref{p1-nondiv}.
For $v\in C^2(\Omega)$, define
\begin{equation}\label{def-L}
 \cL v
 :=\Delta v-\frac{\Hess v(\nabla u,\nabla u)}{W^2}
 =\cL^{ij}v_{i;j}.
\end{equation}
The last equality holds in any local orthonormal frame,
with the notation of Section~\ref{s21}.

\begin{lemma}
\label{lem:ellipticity}
In any orthonormal frame, for every $X\in T_xM$,
with $X_i=\langle X,e_i\rangle_g$,
\[
 \frac{|X|_g^2}{W^2}
 \le\cL^{ij}X_iX_j
 \le|X|_g^2,
 \quad
 \sum_{i=1}^n\cL^{ii}=n-1+W^{-2}\le n.
\]
If $v\in C^2$ near $q\in\Omega$ has a local maximum
at $q$, then $\cL v(q)\le0$.
\end{lemma}
\begin{proof}
The coefficients differ from those of the Laplacian
only in the direction of $\nabla u$.
At a point where $\nabla u\ne0$, choose an orthonormal
basis with $e_1=\frac{\nabla u}{|\nabla u|_g}$. Then
\[
 (\cL^{ij})
 =\operatorname{diag}(W^{-2},1,\ldots,1).
\]
Where $\nabla u=0$, we have $W=1$ and
$(\cL^{ij})=I$.
The asserted bounds and trace formula follow
and are unchanged by orthogonal changes of frame.

At a local maximum $q$, $\Hess v(q)$ is negative
semidefinite. In the above basis,
\[
 \cL v(q)
 =W(q)^{-2}v_{1;1}(q)
  +\sum_{i=2}^n v_{i;i}(q)
 \le0.
\]
\end{proof}

 Thus $\cL$ is elliptic, and it is uniformly elliptic
 on every subdomain where $|\nabla u|_g$ is bounded.
Finally, contracting the Hessian product and chain
rules with $\cL^{ij}$ gives
\begin{equation}\label{eq-L}
 \begin{aligned}
 &\cL(vw)
 =v\cL w+w\cL v+2\cL^{ij}v_iw_j,\\
 &\cL(\psi(v))
 =\psi'(v)\cL v+\psi''(v)\cL^{ij}v_iv_j,
 \end{aligned}
\end{equation}
for $v,w\in C^2$ and $\psi\in C^2$ on an interval
containing the range of $v$.

\subsection{A weighted Bochner inequality}\label{s24}
We combine the classical Bochner identity
\cite[Chapter~9]{Petersen2016}
with the expression for $\cL$ in \eqref{def-L}.
For related computations, see
\cite{Xu2026,XuZhang2017}.

Set
\[
 z=|\nabla u|_g^2,\quad
 \Omega_*=\{x\in\Omega:z(x)>0\}.
\]

\begin{lemma}
      \label{lem:bochner}\rm
	  Let $u\in C^1(\Omega)\cap C^3(\Omega_*)$ satisfy
\eqref{p1-nondiv} on $\Omega_*$, 
and write
$G=Wf(x,\nabla u)$.
Then, on $\Omega_*$,
\[
 \begin{aligned}
 \cL z
 =2|\Hess u|_g^2
     +2\langle\nabla u,\nabla G\rangle_g
     +\frac{|\nabla z|_g^2}{2W^2}-\frac{\langle\nabla z,\nabla u\rangle_g^2}{W^4}
     +2\Ric(\nabla u,\nabla u),
 \end{aligned}
\]
and
\[
 |\Hess u|_g^2
 \ge \frac{G^2}{n}
     +\frac{2+z}{4W^4}|\nabla z|_g^2.
\]
Under the Ricci bound in \eqref{eq-geometry},
on $B_R(p)\cap\Omega_*$ we consequently have
\begin{equation}\label{eq:bochner-ineq}
 \begin{aligned}
 \cL z\ge
 2\langle\nabla u,\nabla G\rangle_g+\frac2nG^2
 +\frac{3+2z}{2W^4}|\nabla z|_g^2-\frac{\langle\nabla z,\nabla u\rangle_g^2}{W^4}
 -2(n-1)\kappa^2z.
 \end{aligned}
\end{equation}

\end{lemma}
\begin{proof}
The classical Bochner identity, with the curvature
convention of Section~\ref{s21}, reads
\[
 \frac12\Delta z
 =|\Hess u|_g^2
  +\langle\nabla u,\nabla\Delta u\rangle_g
  +\Ric(\nabla u,\nabla u).
\]
We express the middle term using $\cL u=G$.

Fix $x\in\Omega_*$ and choose a normal orthonormal
frame at $x$. Metric compatibility and symmetry of
the Hessian give
\begin{equation}\label{eq:grad-z}
 z_i=2u_{i;j}u_j,\quad
 \nabla z=2\nabla_{\nabla u}\nabla u.
\end{equation}
Hence the equation becomes
\[
 \Delta u=G+\frac{u_i z_i}{2W^2}.
\]
All component calculations below are evaluated at $x$.
Since $e_k(W^{-2})=-\frac{z_k}{W^4}$, using \eqref{eq:grad-z},
\[
 \begin{aligned}
 2\langle\nabla u,\nabla\Delta u\rangle_g&=2u_k(\Delta u)_k\\
 &=2u_kG_k
 +\frac{u_k u_{i;k}z_i+u_k u_i z_{i;k}}{W^2}
 -\frac{(u_kz_k)(u_i z_i)}{W^4}\\
 &=2\langle\nabla u,\nabla G\rangle_g
 +\frac{|\nabla z|_g^2}{2W^2}
 +\frac{\Hess z(\nabla u,\nabla u)}{W^2}-\frac{\langle\nabla z,\nabla u\rangle_g^2}{W^4}.
 \end{aligned}
\]
Substituting into the Bochner identity and using
\[
 \cL z
 =\Delta z-\frac{\Hess z(\nabla u,\nabla u)}{W^2}
\]
proves the asserted identity.

For the Hessian bound, choose an orthonormal basis
at $x$ with $e_1=\frac{\nabla u}{\sqrt z}$.
Then
\[
 G=\cL u=\frac{u_{1;1}}{W^2}+\sum_{a=2}^n u_{a;a},
 \quad
 |\nabla z|_g^2=4z\sum_{i=1}^n u_{1;i}^2.
\]
Cauchy's inequality gives
\[
 \frac{G^2}{n}
 \le \frac{u_{1;1}^2}{W^4}
     +\sum_{a=2}^n u_{a;a}^2.
\]
Since $1-W^{-4}=\frac{z(2+z)}{W^4}<1$, it follows that
\[
 \begin{aligned}
 |\Hess u|_g^2-\frac{G^2}{n}
 &\ge
 \sum_{i,j=1}^n u_{i;j}^2
 -\frac{u_{1;1}^2}{W^4}
 -\sum_{a=2}^n u_{a;a}^2\\
 &\ge
 (1-W^{-4})u_{1;1}^2
 +2\sum_{a=2}^n u_{1;a}^2\\
 &\ge
 \frac{z(2+z)}{W^4}
 \sum_{i=1}^n u_{1;i}^2
 =\frac{2+z}{4W^4}|\nabla z|_g^2.
 \end{aligned}
\]
This proves the Hessian bound.

Finally, insert this bound into the identity above
and use
\[
 \Ric(\nabla u,\nabla u)\ge-(n-1)\kappa^2z,
 \quad
 \frac{2+z}{2W^4}+\frac1{2W^2}
 =\frac{3+2z}{2W^4}.
\]
This yields \eqref{eq:bochner-ineq}.
\end{proof}

\subsection{The distance cutoff}\label{s25}
Fix $0<R_0<R$ and $\alpha\ge2$.
Recall that $K$ is defined in \eqref{eq-scales}
of Theorem~\ref{thm-main}.
With $r=d_g(p,\cdot)$, set
\[
 \eta=1-\frac{r^2}{R_0^2},\quad
 \phi=\eta^\alpha
 \quad\text{on }B_{R_0}(p).
\]
Extend $\phi$ continuously to $\overline B_{R_0}(p)$
by setting $\phi=0$ where $r=R_0$.

\begin{lemma}\label{lem:cutoff}\rm
Assume the Hessian comparison in
Definition~\ref{def-compare}, and let $\cL$ be defined by
\eqref{def-L}. The following estimates hold, with constants
depending only on the indicated subscripts.

\begin{enumerate}
\item[(i)]
If
$q\in B_{R_0}(p)\setminus(\{p\}\cup\Cut(p))$,
then $\phi$ is smooth near $q$, and at $q$,
\begin{equation}\label{eq:cutoff-L}
 \frac{\cL\phi}{\phi}
 -\frac{\cL^{ij}\phi_i\phi_j}{\phi^2}
 \ge-C_{n,\alpha}K(R_0,\mu)\phi^{-2/\alpha},
\end{equation}
\begin{equation}\label{eq:cutoff-grad}
 \frac{|\nabla\phi|_g^2}{\phi^2}
 \le C_\alpha K(R_0,\mu)\phi^{-2/\alpha}.
\end{equation}

\item[(ii)]
The function $\phi$ is smooth near $p$, and
\[
 \phi(p)=1,\quad
 \nabla\phi(p)=0,\quad
 \Hess\phi|_p=-\frac{2\alpha}{R_0^2}g_p.
\]
In particular, \eqref{eq:cutoff-L} and
\eqref{eq:cutoff-grad} also hold at $p$.

\item[(iii)]
If $q\in\Cut(p)\cap B_{R_0}(p)$, then, for every
sufficiently small $\varepsilon>0$, the upper support
$r_{q,\varepsilon}$ in Definition~\ref{def-compare}
defines
\[
 \eta_{q,\varepsilon}
 =1-\frac{r_{q,\varepsilon}^2}{R_0^2},\quad
 \phi_{q,\varepsilon}=\eta_{q,\varepsilon}^{\alpha}
\]
on a sufficiently small neighborhood of $q$.
The function $\phi_{q,\varepsilon}$ is smooth and satisfies
\[
 0<\phi_{q,\varepsilon}\le\phi,\quad
 \phi_{q,\varepsilon}(q)=\phi(q).
\]
At $q$,
\begin{equation}\label{eq:supported-cutoff-L}
 \begin{aligned}
 &\frac{\cL\phi_{q,\varepsilon}}{\phi_{q,\varepsilon}}
 -\frac{\cL^{ij}(\phi_{q,\varepsilon})_i
                   (\phi_{q,\varepsilon})_j}
        {\phi_{q,\varepsilon}^2}\ge
 -C_{n,\alpha}
 \left[K(R_0,\mu)
       +\frac{\varepsilon(1+R_0)}{R_0^2}\right]
 \phi(q)^{-2/\alpha},
 \end{aligned}
\end{equation}
and
\begin{equation}\label{eq:supported-cutoff-gradient}
 \frac{|\nabla\phi_{q,\varepsilon}(q)|_g^2}
      {\phi_{q,\varepsilon}(q)^2}
 \le C_\alpha K(R_0,\mu)\phi(q)^{-2/\alpha}.
\end{equation}
If $\varepsilon(1+R_0)\le1+\mu R_0$, then
\eqref{eq:cutoff-L} and \eqref{eq:cutoff-grad}
hold at $q$ with $\phi$ replaced by
$\phi_{q,\varepsilon}$, after increasing $C_{n,\alpha}$.
\end{enumerate}
\end{lemma}

\begin{proof}
We contract the distance Hessian comparison with the
positive coefficients $\cL^{ij}$.
At $p$ we use the smoothness of $r^2$; at a cut point
we use the upper support from
Definition~\ref{def-compare}.

\smallskip
\noindent\emph{Case (i).}
All calculations in this case are evaluated at $q$.
Choose a local orthonormal frame normal at $q$ such that
\[
 (\cL^{ij})(q)
 =\operatorname{diag}(W^{-2},1,\ldots,1),
\]
as in Lemma~\ref{lem:ellipticity}.
By \eqref{eq-comp2} and $|\nabla r|_g=1$,
\[
 \begin{aligned}
 \cL r
 &=W^{-2}r_{1;1}+\sum_{i=2}^n r_{i;i}\\
 &\le c_\mu(r)
 \left[W^{-2}(1-r_1^2)+\sum_{i=2}^n(1-r_i^2)\right]\\
 &\le n c_\mu(r).
 \end{aligned}
\]
The definition of $c_\mu$ gives
\[
 r c_\mu(r)\le1+\mu r;
\]
for $\mu>0$ this follows from $s\coth s\le1+s$,
and for $\mu=0$ it is an equality.
Since
\[
 (r^2)_{i;j}=2r_i r_j+2r r_{i;j},
\]
Lemma~\ref{lem:ellipticity} yields
\[
 \begin{aligned}
 \cL(r^2)
 &=2r\cL r+2\cL^{ij}r_i r_j\\
 &\le2nr c_\mu(r)+2\\
 &\le2(n+1)(1+\mu R_0).
 \end{aligned}
\]
Consequently,
\[
 \cL\eta\ge-2(n+1)K(R_0,\mu),\quad
 |\nabla\eta|_g^2
 =\frac{4r^2}{R_0^4}\le\frac4{R_0^2}.
\]

The Hessian chain rule gives
\[
 \begin{aligned}
 \phi_i
 &=\alpha\eta^{\alpha-1}\eta_i,\\
 \phi_{i;j}
 &=\alpha\eta^{\alpha-1}\eta_{i;j}
   +\alpha(\alpha-1)\eta^{\alpha-2}\eta_i\eta_j.
 \end{aligned}
\]
Thus, 
\[
 \begin{aligned}
 \frac{\cL\phi}{\phi}
   -\frac{\cL^{ij}\phi_i\phi_j}{\phi^2}
 &=\alpha\frac{\cL\eta}{\eta}
   -\alpha\frac{\cL^{ij}\eta_i\eta_j}{\eta^2}\\
 &\ge-\frac{2\alpha(n+1)K(R_0,\mu)}{\eta}
       -\frac{4\alpha}{R_0^2\eta^2}\\
 &\ge-C_{n,\alpha}K(R_0,\mu)\eta^{-2}.
 \end{aligned}
\]
Here we used $0<\eta\le1$ and $R_0^{-2}\le K(R_0,\mu)$.
Since $\eta^{-2}=\phi^{-2/\alpha}$, this proves
\eqref{eq:cutoff-L}. Similarly,
\[
 \frac{|\nabla\phi|_g^2}{\phi^2}
 =\alpha^2\frac{|\nabla\eta|_g^2}{\eta^2}
 \le\frac{4\alpha^2}{R_0^2}\eta^{-2}
 \le4\alpha^2K(R_0,\mu)\phi^{-2/\alpha},
\]
which proves \eqref{eq:cutoff-grad}.

\smallskip
\noindent\emph{Case (ii).}
Choose normal coordinates
$(x^1,\ldots,x^n)$ centered at $p$.
In a normal neighborhood, the Gauss lemma gives
\[
 r^2(x)=\sum_{i=1}^n(x^i)^2;
\]
see \cite[Chapter~5]{Petersen2016}.
At $p$,
\[
 g_{ij}=\delta_{ij},\quad
 \Gamma_{ij}^k=0,\quad
 \partial_i(r^2)=0,\quad
 \partial_i\partial_j(r^2)=2\delta_{ij}.
\]
Hence
\[
 \dd(r^2)|_p=0,\quad
 \Hess(r^2)|_p=2g_p.
\]
The chain rule therefore gives
\[
 \phi(p)=1,\quad
 \nabla\phi(p)=0,\quad
 \Hess\phi|_p
 =\alpha\Hess\eta|_p
 =-\frac{2\alpha}{R_0^2}g_p.
\]
By Lemma~\ref{lem:ellipticity},
\[
 \left(\frac{\cL\phi}{\phi}
   -\frac{\cL^{ij}\phi_i\phi_j}{\phi^2}\right)(p)
 =-\frac{2\alpha}{R_0^2}
   \bigl(n-1+W(p)^{-2}\bigr)
 \ge-\frac{2\alpha n}{R_0^2}.
\]
Together with $\phi(p)=1$ and $\nabla\phi(p)=0$,
this proves both estimates at $p$.

\smallskip
\noindent\emph{Case (iii).}
Let $r_{q,\varepsilon}$ satisfy
\eqref{eq-supp} and \eqref{eq-comp3}.
Since $r_{q,\varepsilon}(q)=r(q)\in(0,R_0)$,
we may restrict its neighborhood so that
$0<r_{q,\varepsilon}<R_0$.
The function $t\mapsto\left(1-\frac{t^2}{R_0^2}\right)^\alpha$ is decreasing
on $(0,R_0)$, so
\[
 0<\phi_{q,\varepsilon}\le\phi,\quad
 \phi_{q,\varepsilon}(q)=\phi(q).
\]

We first verify the gradient bound stated in
Remark~\ref{rem1}.
For every unit vector $X\in T_qM$ and sufficiently small
$t$, the support property and the triangle inequality give
\[
 r_{q,\varepsilon}(\exp_q(tX))
 -r_{q,\varepsilon}(q)
 \ge r(\exp_q(tX))-r(q)
 \ge-|t|.
\]
Taking one-sided derivatives at $t=0$ yields
\[
 -1\le\dd r_{q,\varepsilon}(X)\le1.
\]
Since $X$ is arbitrary,
\[
 |\nabla r_{q,\varepsilon}(q)|_g\le1,\quad
 0\le
 (\dd r_{q,\varepsilon}\otimes\dd r_{q,\varepsilon})_q
 \le g_q.
\]

The same frame calculation as in case~(i), with
\eqref{eq-comp3} replacing \eqref{eq-comp2}, gives
\[
 \cL r_{q,\varepsilon}(q)
 \le n c_\mu(r(q))+2n\varepsilon.
\]
Now apply the identities for $\cL(r^2)$ and
$\nabla\eta$ in case~(i), replacing
\[
 r\ \text{by }r_{q,\varepsilon},\quad
 \eta\ \text{by }\eta_{q,\varepsilon}.
\]
Using $r_{q,\varepsilon}(q)=r(q)<R_0$ and the gradient
bound above, we obtain
\[
 \begin{aligned}
 &\cL(r_{q,\varepsilon}^2)(q)
 \le2(n+1)(1+\mu R_0)+4n\varepsilon R_0,\\
 &\cL\eta_{q,\varepsilon}(q)
 \ge-2(n+1)K(R_0,\mu)
       -\frac{4n\varepsilon}{R_0},\\
 &|\nabla\eta_{q,\varepsilon}(q)|_g^2
 \le\frac4{R_0^2}.
 \end{aligned}
\]

Finally, the identities for $\frac{\cL\phi}{\phi}
   -\frac{\cL^{ij}\phi_i\phi_j}{\phi^2}$ and
$\frac{|\nabla\phi|_g^2}{\phi^2}$ in case~(i) apply with
$(\eta,\phi)$ replaced by
$(\eta_{q,\varepsilon},\phi_{q,\varepsilon})$.
Since
\[
 \eta_{q,\varepsilon}(q)=\eta(q),\quad
 \phi_{q,\varepsilon}(q)=\phi(q),\quad
 \frac{\varepsilon}{R_0}
 \le\frac{\varepsilon(1+R_0)}{R_0^2},
\]
the preceding bounds prove
\eqref{eq:supported-cutoff-L} and
\eqref{eq:supported-cutoff-gradient}.
If $\varepsilon(1+R_0)\le1+\mu R_0$, the error term
is bounded by $K(R_0,\mu)$, proving the last assertion.
\end{proof}

\section{Proof of the main theorem}\label{s3}
Fix $0<R_0<R$ and $\alpha\ge2$, and let $\phi$ be the
cutoff from Section~\ref{s25}.
Let $\mathcal T$ be a positive $C^2$ function on an open
interval containing $u(\overline B_{R_0}(p))$, and let
\[
 F\in C^2((0,\infty))\cap C([0,\infty)),\quad
 F(0)=0,\quad F,F'>0\ \text{on }(0,\infty).
\]
We use the test function
\begin{equation}\label{eq-test}
 \Psi=\mathcal T(u)F(z)\phi,
\end{equation}
where $\mathcal T$ and $F$ will be specified in each case.
Primes on $\mathcal T$ and $F$ denote derivatives with
respect to their scalar arguments.
\subsection{The maximum inequality}
We may assume $z(p)>0$, since otherwise the asserted
gradient estimates are immediate.
Let $\Omega_0(p)$ be the connected component of
$B_{R_0}(p)\cap\Omega_*$ containing $p$.
By completeness and the Hopf--Rinow theorem,
$\overline{\Omega_0(p)}\subset\overline B_{R_0}(p)$
is compact.

\begin{lemma}
\label{lem:maximum-localization}\rm
The function $\Psi$ attains its maximum on
$\overline{\Omega_0(p)}$ at a point $q\in\Omega_0(p)$,
with
\[
 0<\Psi(p)\le\Psi(q).
\]
At every such point,
\[
 \nabla\log\Psi(q)=0,\quad
 \cL(\log\Psi)(q)\le0,
\]
 where, if $q\in\Cut(p)$, these relations are understood
with $\phi$ replaced by the smooth lower support
$\phi_{q,\varepsilon}$ defined in
Lemma~\ref{lem:cutoff}\textup{(iii)},
for every sufficiently small $\varepsilon>0$.
\end{lemma}
\begin{proof}
The function $\Psi$ is continuous on
$\overline{\Omega_0(p)}$ and $\Psi(p)>0$.
Notice that
\[\partial \Omega_{0}(p)\subset \partial B_{R_{0}}(p) \cup \left( B_{R_{0}}(p)\cap \{z=0\}\right).\]
On the first set $\phi=0$, while on the second set $F(z)=0$.
Hence $\Psi=0$ on $\partial\Omega_0(p)$. So every maximum point lies in $\Omega_0(p)$.

Since $z(q)>0$, we have $u\in C^3$ and $z\in C^2$
near $q$.
If $q\notin\Cut(p)$, including $q=p$,
the function $\log\Psi$ is $C^2$ near $q$.
If $q\in\Cut(p)$, then near $q$,
\[
 \mathcal T(u)F(z)\phi_{q,\varepsilon}
 \le\Psi
 \le\Psi(q)
 =\mathcal T(u(q))F(z(q))\phi_{q,\varepsilon}(q).
\]
The corresponding logarithm therefore has a local
maximum at $q$.
Lemma~\ref{lem:ellipticity} gives the differential
conclusions.
\end{proof}

For each choice of $\mathcal T$ and $F$, choose
$q\in\Omega_0(p)$ as in
Lemma~\ref{lem:maximum-localization}.
At a cut point, take $\varepsilon>0$ sufficiently small
that
\[
 \varepsilon(1+R_0)\le1+\mu R_0,
\]
and use the smooth lower support $\phi_{q,\varepsilon}$
defined in Lemma~\ref{lem:cutoff}\textup{(iii)}.
In all differentiated expressions at such a point,
$\phi$ is understood as $\phi_{q,\varepsilon}$;
its value at $q$ remains $\phi(q)$.
The estimates \eqref{eq:cutoff-L} and
\eqref{eq:cutoff-grad} then hold uniformly,
after increasing $C_{n,\alpha}$ once.

All quantities below are evaluated at $q$.
The functions $\mathcal T,F$ and their derivatives
are evaluated at $u(q),z(q)$, respectively, and
$f,\nabla^h f,\nabla^v f$ at $(q,\nabla u(q))$.
Set
\begin{equation}\label{eq-PQ}
 \begin{aligned}
 &P_F
 =\frac{F''}{F}-\frac{(F')^2}{F^2}
   +\frac{3+2z}{2W^4}\frac{F'}{F},\\
 &Q_F
 =-\frac{F''}{F}+\frac{(F')^2}{F^2}
   -\frac{1}{W^2}\frac{F'}{F},
 \end{aligned}
\end{equation}
and write
\begin{equation}\label{eq-I}
 \begin{aligned}
 &I_1
 =2\frac{F'}{F}W
   \langle\nabla^h f,\nabla u\rangle_g,\\ 
 &I_2
 =\frac{F'}{F}W
   \langle\nabla^v f,\nabla z\rangle_g,\\
 &I_3
 =\frac{2}{n}\frac{F'}{F}W^2f^2,\\
 &I_4
 =\frac{\mathcal T'}{\mathcal T}\frac{f}{W},\\
 &I_5
 =\left(
     \frac{\mathcal T''}{\mathcal T}
     -\frac{(\mathcal T')^2}{\mathcal T^2}
   \right)\frac{z}{W^2},\\
 &I_6
 =P_F|\nabla z|_g^2,\\
 &I_7
 =Q_F\frac{\langle\nabla z,\nabla u\rangle_g^2}{W^2},\\
 &I_8
 =\frac{f}{W}
   \left\langle\nabla u,\frac{\nabla\phi}{\phi}
   \right\rangle_g,\\
 &I_9
 =2(n-1)\kappa^2z\frac{F'}{F},\\
 &I_{10}
 =C_{n,\alpha}K(R_0,\mu)\phi^{-\frac{2}{\alpha}}.
 \end{aligned}
\end{equation}
Here $C_{n,\alpha}$ is the uniform constant supplied
by Lemma~\ref{lem:cutoff}.

\begin{lemma}\label{lem:master}\rm
At the maximum point $q$ chosen above,
\begin{equation}\label{eq-gradz-general}
 \nabla z
 =-\frac{F}{F'}
 \left(
   \frac{\mathcal T'}{\mathcal T}\nabla u
   +\frac{\nabla\phi}{\phi}
 \right),
\end{equation}
and
\begin{equation}\label{eq:master}
 I_1+I_2+I_3+I_4+I_5+I_6+I_7
 \le I_8+I_9+I_{10}.
\end{equation}
\end{lemma}

\begin{proof}
We combine the logarithmic maximum argument
\cite{Wang1998} with the chain rule and
the weighted Bochner inequality. All calculations below are evaluated at $q$.

At $q$,
\[
 0=\nabla\log\Psi(q)
 =
	\left(\frac{\mathcal T'(u)}{\mathcal T (u)}\nabla u
       +\frac{F'(z)}{F(z)}\nabla z+\frac{\nabla\phi}{\phi}\right).
\]
This proves \eqref{eq-gradz-general}.

By Lemma~\ref{lem:maximum-localization}, we have
\begin{equation}\label{eq-second-deriv1}
 0\geq (\cL\log\Psi)(q)=\cL\left(\log \mathcal T(u)+\log F(z)+\log \phi\right).
\end{equation}

Choose a local orthonormal frame normal at $q$.
From the definition of $\cL^{ij}$,
\[
 \begin{aligned}
 &\cL^{ij}u_i u_j
 =z-\frac{z^2}{W^2}
 =\frac{z}{W^2},\\
 &\cL^{ij}z_i z_j
 =|\nabla z|_g^2
   -\frac{\langle\nabla z,\nabla u\rangle_g^2}{W^2}.
 \end{aligned}
\]
Applying the composition rule \eqref{eq-L} to the three terms in \eqref{eq-second-deriv1}, we obtain
\begin{equation}\label{eq-second-deriv2}
\begin{aligned}
&\cL\log\mathcal T(u)
 =\frac{\mathcal T'}{\mathcal T}\cL u
   +\left(\frac{\mathcal T''}{\mathcal T}
             -\frac{(\mathcal T')^2}{\mathcal T^2}\right)
                      \frac z{W^2},\\
&\cL\log F(z)
 =\frac{F'}F\cL z
   +\left(\frac{F''}{F}-\frac{(F')^2}{F^2}\right)
     \left(|\nabla z|_g^2
       -\frac{\langle\nabla z,\nabla u\rangle_g^2}{W^2}\right),\\
&\cL\log\phi
  =\frac{\cL\phi}{\phi}-\frac{\cL^{ij}\phi_{i}\phi_{j}}{\phi^2}.
\end{aligned}
\end{equation}

Next, we estimate $\cL z$. Since $G=\cL u=Wf(x,\nabla u)$, Lemma~\ref{lem:chain} gives
\[
\begin{aligned}
2\langle\nabla u,\nabla G\rangle_g
 &=2f\langle\nabla u,\nabla W\rangle_g
      +2W\langle\nabla^h f,\nabla u\rangle_g+2W\langle\nabla^v f,\nabla_{\nabla u}\nabla u\rangle_g.
\end{aligned}
\]
Substituting $\nabla W=\frac{\nabla z}{2W}$ and \eqref{eq:grad-z}
into the first and third terms yields
\begin{equation}\label{eq-UG}
\begin{aligned}
2\langle\nabla u,\nabla G\rangle_g
 &=\frac fW\langle\nabla u,\nabla z\rangle_g
       +2W\langle\nabla^h f,\nabla u\rangle_g+W\langle\nabla^v f,\nabla z\rangle_g.
\end{aligned}
\end{equation}
Combining \eqref{eq-UG} and the weighted Bochner inequality
\eqref{eq:bochner-ineq}, 
\begin{equation}\label{eq-Lz}
\begin{aligned}
\cL z\ge{}&\frac fW\langle\nabla u,\nabla z\rangle_g
 +2W\langle\nabla^h f,\nabla u\rangle_g
 +W\langle\nabla^v f,\nabla z\rangle_g\\
&+\frac2n W^2f^2
 +\frac{3+2z}{2W^4}|\nabla z|_g^2
 -\frac{\langle\nabla z,\nabla u\rangle_g^2}{W^4}
 -2(n-1)\kappa^2z.
\end{aligned}
\end{equation}

Combining \eqref{eq-second-deriv1}, \eqref{eq-second-deriv2},
\eqref{eq-Lz}, and the cutoff estimate \eqref{eq:cutoff-L},
and using $\cL u=Wf$, we obtain
\begin{equation}\label{eq:master-before-first}
	\begin{aligned}
		0&\ge\frac{\mathcal {T}'}{\mathcal {T}} Wf
		+\left(\frac{\mathcal {T}''}{\mathcal {T}}-\frac{(\mathcal {T}')^2}{\mathcal {T}^2}\right)\frac z{W^2}\\
		&\quad +\frac{F'}F\frac fW\langle \nabla u,\nabla z\rangle_g
		+2\frac{F'}F W
			\langle \nabla^h f,\nabla u\rangle_g
		+\frac{F'}F W\langle\nabla^vf,\nabla z\rangle_g\\
		&+\frac2n\frac{F'}F W^2f^2
		-2(n-1)\kappa^2z\frac{F'}F
		-C_{n,\alpha}K(R_{0},\mu)\phi^{-2/\alpha}\\
		&+\left[
		\frac{F''}{F}-\frac{(F')^2}{F^2}
		+\frac{3+2z}{2W^4}\frac{F'}F
		\right]|\nabla z|_g^2\\
		&+\left[
		-\frac{F''}{F}+\frac{(F')^2}{F^2}
		-\frac1{W^2}\frac{F'}F
		\right]
		\frac{\langle\nabla z,\nabla u\rangle_g^2}{W^2}.
	\end{aligned}
\end{equation}
The last two brackets are precisely $P_F$ and $Q_F$ from
\eqref{eq-PQ}.

Finally, we combine the two terms involving the undifferentiated
$f$.  From \eqref{eq-gradz-general}, at $q$, 
\[\begin{aligned}
	 \frac{F'}F\langle \nabla u,\nabla z\rangle_g&=\frac{F'}F\left\langle \nabla u,-\frac{F}{F'}\left(\frac{\mathcal {T}'}{\mathcal {T}} \nabla u+\frac{\nabla\phi}{\phi}\right)\right\rangle_g
	=-\frac{\mathcal {T}'}{\mathcal {T}} z-\left\langle \nabla u,\frac{\nabla \phi}{\phi}\right\rangle_g.
\end{aligned}
\]
Therefore
\begin{equation}\label{eq-f1}
	\begin{aligned}
 \frac{\mathcal {T}'}{\mathcal {T}} Wf+\frac{F'}F\frac fW\langle \nabla u,\nabla z\rangle_g
 &=\frac{\mathcal {T}'}{\mathcal {T}} f\left(W-\frac zW\right)
   -\frac fW\left\langle \nabla u,\frac{\nabla \phi}{\phi}\right\rangle_g\\
 &=\frac{\mathcal {T}'}{\mathcal {T}}\frac fW-\frac fW\left\langle \nabla u,\frac{\nabla \phi}{\phi}\right\rangle_g,
 \end{aligned}
\end{equation}
because $W^2-z=1$.  Substituting \eqref{eq-f1} in
\eqref{eq:master-before-first}, we have \eqref{eq:master}.

\end{proof}

We now prove the five assertions of Theorem~\ref{thm-main}.
We retain the maximum-point and cutoff conventions above,
and use $L$, $K$ and $\Lambda$ as defined in \eqref{eq-scales}.
The elementary inequalities
\begin{equation}\label{eq:log-elementary}
	\begin{aligned}
 \frac{z}{W^2}\le\log(1+z)\le z,\quad \left(1+\frac{\log(1+z)}{2}\right)\frac{z}{W^2}
 \ge\frac{1}{2}\log(1+z),
	\end{aligned}
\end{equation}
will be used repeatedly. The second inequality follows from
\[
 \left(1+\frac{\log(1+z)}{2}\right)\frac{z}{W^2}
 -\frac{1}{2}\log(1+z)
 =
 \frac{z-\frac{1}{2}\log(1+z)}{1+z}
 \ge0.
\]

For parts \textup{(iii)}--\textup{(v)}, 
the conclusion is immediate if $L=0$, and is vacuous if
$L=\infty$. We therefore assume $0<L<\infty$ in these
three parts. In particular,
\[
 0\le u-\inf\nolimits_{B_R(p)}u\le L,\quad
 0\le\sup\nolimits_{B_R(p)}u-u\le L.
\]
The letter $C$ denotes a positive constant that may change
from line to line and depends only on the indicated parameters.
All such constants are independent of $R_0$.

\subsection{Proof of part \textup{(i)}}

Under \eqref{A1}, we use the positive term involving $f^2$
to absorb the terms containing derivatives of $f$.
Choose
\[
 \mathcal T\equiv1,\quad F(z)=z,\quad
 \alpha=\max\left\{2,\frac{2}{\theta}\right\}.
\]
At the maximum point $q$, \eqref{eq-gradz-general} gives
\[
 \nabla z=-z\frac{\nabla\phi}{\phi}.
\]
For this choice of $F$, \eqref{eq-PQ} becomes
\begin{equation}\label{eq-a-PQz}
 \frac{F'}{F}=\frac{1}{z},\quad
 P_F=-\frac{2+z}{2z^2W^4},\quad
 Q_F=\frac{1}{z^2W^2}.
\end{equation}

For arbitrary $\varepsilon_1,\varepsilon_2,\varepsilon_3>0$,
\eqref{eq-f}, \eqref{eq-a-PQz} and Young's inequality give
\[
\begin{aligned}
 I_1
 &\ge-\frac{2\sigma_hW|f|}{\sqrt z}\ge-2\varepsilon_1\frac{W^2f^2}{z}
       -\frac{\sigma_h^2}{2\varepsilon_1},\\
 I_2
 &=-W\left\langle\nabla^v f,
                    \frac{\nabla\phi}{\phi}\right\rangle_g\ge-\varepsilon_2W^2|\nabla^v f|_g^2
       -\frac{1}{4\varepsilon_2}
           \left|\frac{\nabla\phi}{\phi}\right|_g^2,\\
 I_3
 &=\frac{2}{n}\frac{W^2f^2}{z},\quad I_4=0,\quad I_5=0,\\
 I_6
 &=-\frac{2+z}{2W^4}
          \left|\frac{\nabla\phi}{\phi}\right|_g^2
 \ge-\left|\frac{\nabla\phi}{\phi}\right|_g^2,\quad I_7
 =\frac{1}{W^4}
       \left\langle\nabla u,
                    \frac{\nabla\phi}{\phi}\right\rangle_g^2
 \ge0,\\
 I_8
 &\le\frac{|f|\sqrt z}{W}
              \left|\frac{\nabla\phi}{\phi}\right|_g\le\varepsilon_3\frac{W^2f^2}{z}
       +\frac{1}{4\varepsilon_3}\frac{z^2}{W^4}
              \left|\frac{\nabla\phi}{\phi}\right|_g^2\le\varepsilon_3\frac{W^2f^2}{z}
       +\frac{1}{4\varepsilon_3}
              \left|\frac{\nabla\phi}{\phi}\right|_g^2,\\
 I_9&=2(n-1)\kappa^2,\quad I_{10}
 =C_{n,\alpha}K(R_0,\mu)\phi^{-\frac{2}{\alpha}}.
\end{aligned}
\]
Substitution into \eqref{eq:master} yields
\[
\begin{aligned}
 &\left(\frac{2}{n}-2\varepsilon_1-\varepsilon_3\right)
                  \frac{W^2f^2}{z}
       -\varepsilon_2W^2|\nabla^v f|_g^2\\
 &\quad\le
 \left(1+\frac{1}{4\varepsilon_2}
          +\frac{1}{4\varepsilon_3}\right)
             \left|\frac{\nabla\phi}{\phi}\right|_g^2
 +\frac{\sigma_h^2}{2\varepsilon_1}
 +2(n-1)\kappa^2
 +C_{n,\alpha}K(R_0,\mu)\phi^{-\frac{2}{\alpha}}.
\end{aligned}
\]
We now choose
\[
 \varepsilon_1=\frac{1}{4n},\quad
 \varepsilon_2=\frac{m_1}{n},\quad
 \varepsilon_3=\frac{1}{2n}.
\]
Then
\[
 \frac{2}{n}-2\varepsilon_1-\varepsilon_3
 =\frac{1}{n},
\]
so the preceding inequality becomes
\[
\begin{aligned}
 \frac{W^2}{nz}
       \left(f^2-m_1z|\nabla^v f|_g^2\right)
 &\le
 \left(1+\frac{n}{4m_1}+\frac{n}{2}\right)
              \left|\frac{\nabla\phi}{\phi}\right|_g^2\\
 &\quad+2n\sigma_h^2+2(n-1)\kappa^2
       +C_{n,\alpha}K(R_0,\mu)\phi^{-\frac{2}{\alpha}}.
\end{aligned}
\]
By \eqref{A1} and \eqref{eq:cutoff-grad},
\[
 \frac{m_2}{n}z^\theta
 \le\frac{m_2}{n}W^2z^{\theta-1}
 \le C\Lambda(R_0,\mu,\kappa,\sigma_h)
                      \phi^{-\frac{2}{\alpha}}.
\]

Since $\phi(p)=1$ and $q$ maximizes $z\phi$,
\[
\begin{aligned}
 z(p)
 &\le z(q)\phi(q)\\
 &\le C\left[\Lambda(R_0,\mu,\kappa,\sigma_h)\right]^{
                \frac{1}{\theta}}
          \phi(q)^{1-\frac{2}{\alpha\theta}}\\
 &\le C\left[\Lambda(R_0,\mu,\kappa,\sigma_h)\right]^{
                \frac{1}{\theta}}.
\end{aligned}
\]
Taking square roots proves \eqref{eq:a-grad} at radius $R_0$,
with $C=C(n,m_1,m_2,\theta)$.

\subsection{Proof of part \textup{(ii)}}

The logarithmic factors in \eqref{A2} lead us to choose
\[
 \mathcal T\equiv1,\quad F(z)=\log(1+z),\quad
 \alpha=\max\left\{
   2,\frac{2}{2\theta-1},\frac{1}{\theta-1}
 \right\}.
\]
At $q$,
\[
 \nabla z=-W^2\log(1+z)\frac{\nabla\phi}{\phi}.
\]
The corresponding coefficients are
\begin{equation}\label{eq-b-PQlog}
\begin{aligned}
 \frac{F'}{F}
 &=\frac{1}{W^2\log(1+z)},\quad
 \frac{F''}{F}
 =-\frac{1}{W^4\log(1+z)},\\
 P_F
 &=\frac{\log(1+z)-2W^2}
          {2W^6\log^2(1+z)},\quad
 Q_F=\frac{1}{W^4\log^2(1+z)}.
\end{aligned}
\end{equation}
For arbitrary $\varepsilon_1,\varepsilon_2>0$,
\eqref{eq-f}, \eqref{eq:log-elementary}, \eqref{eq-b-PQlog}, and
Young's inequality imply
\[
\begin{aligned}
 I_1
 &\ge-\frac{2\sigma_h|f|\sqrt z}
              {W\log(1+z)}\ge-2\varepsilon_1\frac{f^2}{\log(1+z)}
       -\frac{\sigma_h^2}{2\varepsilon_1}
                    \frac{z}{W^2\log(1+z)}\\
 &\ge-2\varepsilon_1\frac{f^2}{\log(1+z)}
       -\frac{\sigma_h^2}{2\varepsilon_1},\\
 I_2
 &=-W\left\langle\nabla^v f,
                    \frac{\nabla\phi}{\phi}\right\rangle_g\ge-\varepsilon_2W^2|\nabla^v f|_g^2
       -\frac{1}{4\varepsilon_2}
              \left|\frac{\nabla\phi}{\phi}\right|_g^2,\\
 I_3
 &=\frac{2}{n}\frac{f^2}{\log(1+z)},\quad I_4=0,\quad I_5=0,\\
 I_6
 &=\left(-1+\frac{\log(1+z)}{2W^2}\right)
              \left|\frac{\nabla\phi}{\phi}\right|_g^2
 \ge-\left|\frac{\nabla\phi}{\phi}\right|_g^2,\quad I_7
 =\frac{1}{W^2}
       \left\langle\nabla u,
                    \frac{\nabla\phi}{\phi}\right\rangle_g^2
 \ge0,\\
 I_8
 &\le |f|\sqrt{\frac{z}{W^2}}
                 \left|\frac{\nabla\phi}{\phi}\right|_g
 \le |f|\left|\frac{\nabla\phi}{\phi}\right|_g,\\
 I_9
 &=2(n-1)\kappa^2\frac{z}{W^2\log(1+z)}
 \le2(n-1)\kappa^2,\quad I_{10}
 =C_{n,\alpha}K(R_0,\mu)\phi^{-\frac{2}{\alpha}}.
\end{aligned}
\]
Therefore, \eqref{eq:master} gives
\[
\begin{aligned}
 &\left(\frac{2}{n}-2\varepsilon_1\right)
                    \frac{f^2}{\log(1+z)}
       -\varepsilon_2W^2|\nabla^v f|_g^2\\
 &\quad\le
 \left(1+\frac{1}{4\varepsilon_2}\right)
              \left|\frac{\nabla\phi}{\phi}\right|_g^2
 +|f|\left|\frac{\nabla\phi}{\phi}\right|_g +
 \frac{\sigma_h^2}{2\varepsilon_1}
 +2(n-1)\kappa^2
 +C_{n,\alpha}K(R_0,\mu)\phi^{-\frac{2}{\alpha}}.
\end{aligned}
\]
Choose
\[
 \varepsilon_1=\frac{1}{2n},\quad
 \varepsilon_2=\frac{m_1}{n}.
\]
The left-hand side then equals
\[
 \frac{f^2-m_1W^2\log(1+z)|\nabla^v f|_g^2}
      {n\log(1+z)}.
\]
Using \eqref{A2} and Lemma \ref{lem:cutoff}
we obtain
\[
\begin{aligned}
 \log^{2\theta-1}(1+z)
 \le C\Big(
 &\Lambda(R_0,\mu,\kappa,\sigma_h)
                   \phi^{-\frac{2}{\alpha}}+\log^\theta(1+z)
    \left[\Lambda(R_0,\mu,\kappa,\sigma_h)\right]^{
                        \frac{1}{2}}
                   \phi^{-\frac{1}{\alpha}}
 \Big).
\end{aligned}
\]

At least one term on the right is at least half of
the left-hand side. Since $\theta>1$, this implies
\[
\begin{aligned}
 \log(1+z)\le C\max\Big\{
 &\left[\Lambda(R_0,\mu,\kappa,\sigma_h)\right]^{
                       \frac{1}{2\theta-1}}
       \phi^{-\frac{2}{\alpha(2\theta-1)}},\ \ \left[\Lambda(R_0,\mu,\kappa,\sigma_h)\right]^{
                       \frac{1}{2(\theta-1)}}
       \phi^{-\frac{1}{\alpha(\theta-1)}}
 \Big\}.
\end{aligned}
\]
Our choice of $\alpha$ makes both powers of $\phi$
nonnegative after multiplication by $\phi$.
The maximality of $\log(1+z)\phi$ therefore yields
\[
\begin{aligned}
 \log(1+z(p))
 &\le\log(1+z(q))\phi(q)\\
 &\le C\max\left\{
 \left[\Lambda(R_0,\mu,\kappa,\sigma_h)\right]^{
                        \frac{1}{2(\theta-1)}},\ \ 
 \left[\Lambda(R_0,\mu,\kappa,\sigma_h)\right]^{
                        \frac{1}{2\theta-1}}
 \right\}.
\end{aligned}
\]

If $\Lambda(R_0,\mu,\kappa,\sigma_h)\le1$,  the second
power dominates, and using $e^{Ct}-1\le Ce^C t$ for $0\le t\le1$, exponentiation gives
\[
 z(p)\le
 C\left[\Lambda(R_0,\mu,\kappa,\sigma_h)\right]^{
                       \frac{1}{2\theta-1}}.
\]
If $\Lambda(R_0,\mu,\kappa,\sigma_h)\ge1$, the first
power dominates. Hence, in both cases,
\[
\begin{aligned}
 z(p)\le C\max\Big\{
 &\exp\!\left(
 C\left[\Lambda(R_0,\mu,\kappa,\sigma_h)\right]^{
                      \frac{1}{2(\theta-1)}}
 \right)-1,\ \ \left[\Lambda(R_0,\mu,\kappa,\sigma_h)\right]^{
                      \frac{1}{2\theta-1}}
 \Big\},
\end{aligned}
\]
where $C=C(n,m_1,m_2,m_3,\theta)$. This proves \eqref{eq:b-grad} at radius $R_0$.

\subsection{Proof of part \textup{(iii)}}

Condition \eqref{A3} prescribes the modulus of $f$
as a function of the fibre norm. Since this modulus
is positive off the zero section, the sign of $f$
is locally constant there.

To compute its derivatives, choose normal coordinates
at $x$. At $(x,\xi)$ with $\xi\ne0$, the first derivatives
of the metric and the Christoffel symbols vanish at $x$.
Differentiating \eqref{A3} in these bundle coordinates gives
\[
 \frac{\partial f}{\partial x^i}(x,\xi)=0,\quad
 \frac{\partial f}{\partial\xi^i}(x,\xi)
 =\frac{2\theta f(x,\xi)\xi_i}
        {(1+|\xi|_g^2)\log(1+|\xi|_g^2)}.
\]
Returning to the horizontal and vertical derivatives,
we obtain
\[
 \nabla^h f=0,\quad
 \nabla^v f(x,\xi)
 =\frac{2\theta f(x,\xi)}
        {(1+|\xi|_g^2)\log(1+|\xi|_g^2)}\,\xi.
\]

On the connected set $\Omega_0(p)$, continuity gives
a fixed $\delta\in\{-1,1\}$ such that
\[
 f(x,\nabla u(x))
 =\delta m_1\log^\theta(1+z(x)).
\]
Thus, along this component,
\[
 \nabla^h f=0,\quad
 \nabla^v f
 =\frac{2\delta m_1\theta}{W^2}
       \log^{\theta-1}(1+z)\nabla u.
\]

Fix $0<\beta<1$ and choose
\[
 \alpha=\max\left\{
  2,\frac{2}{\theta},\frac{4}{1-\beta}
 \right\},\quad F(z)=z,
\]
together with
\[
 \mathcal T(u)=
 \begin{cases}
 \left(L+1+u-\inf\nolimits_{B_R(p)}u\right)^{-\frac{1}{b}},
       &\delta=1,\\
 \left(L+1+\sup\nolimits_{B_R(p)}u-u\right)^{-\frac{1}{b}},
       &\delta=-1,
 \end{cases}
\]
where $b>0$ will be fixed below.
The choice of sign ensures that the contribution
of $\nabla^v f$ involving $\mathcal T'$ is positive.
Moreover,
\[
 [2(L+1)]^{-\frac{1}{b}}
 \le\mathcal T\le(L+1)^{-\frac{1}{b}},
\]
and
\[
 \frac{\mathcal T'}{\mathcal T}
 =-\frac{\delta}{b}\mathcal T^b,\quad
 \frac{\mathcal T''}{\mathcal T}
   -\left(\frac{\mathcal T'}{\mathcal T}\right)^2
 =b\left(\frac{\mathcal T'}{\mathcal T}\right)^2.
\]
In particular,
\[
 -\delta\frac{\mathcal T'}{\mathcal T}
 =\left|\frac{\mathcal T'}{\mathcal T}\right|
 \ge\frac{1}{2b(L+1)}.
\]

At the maximum point $q$,
\[
 \nabla z
 =-z\left(
      \frac{\mathcal T'}{\mathcal T}\nabla u
      +\frac{\nabla\phi}{\phi}
     \right).
\]

For arbitrary
$\varepsilon_1,\varepsilon_2,\varepsilon_3,\varepsilon_4>0$, \eqref{eq-a-PQz} and Young's inequality give
\[
\begin{aligned}
 I_1&=0,\\
 I_2
 &=2m_1\theta
       \left|\frac{\mathcal T'}{\mathcal T}\right|
       \frac{z}{W}\log^{\theta-1}(1+z)-\frac{2\delta m_1\theta}{W}
       \log^{\theta-1}(1+z)
       \left\langle\nabla u,
                    \frac{\nabla\phi}{\phi}\right\rangle_g\\
 &\ge
 2m_1\theta
       \left|\frac{\mathcal T'}{\mathcal T}\right|
       \frac{z}{W}\log^{\theta-1}(1+z)
 -\varepsilon_1\frac{z}{W^2}\log^{2\theta-2}(1+z)-\frac{m_1^2\theta^2}{\varepsilon_1}
              \left|\frac{\nabla\phi}{\phi}\right|_g^2,\\
 I_3
 &=\frac{2m_1^2}{n}\frac{W^2}{z}
                          \log^{2\theta}(1+z),\\
 I_4
 &=-m_1\left|\frac{\mathcal T'}{\mathcal T}\right|
                         \frac{\log^\theta(1+z)}{W}\ge-\varepsilon_2\frac{W^2}{z}\log^{2\theta}(1+z)
       -\frac{m_1^2}{4\varepsilon_2}
          \left(\frac{\mathcal T'}{\mathcal T}\right)^2
          \frac{z}{W^4},\\
 I_5
 &=b\left(\frac{\mathcal T'}{\mathcal T}\right)^2
                          \frac{z}{W^2},\quad
I_6
 =-\frac{2+z}{2W^4}
       \left|
         \frac{\mathcal T'}{\mathcal T}\nabla u
         +\frac{\nabla\phi}{\phi}
       \right|_g^2,\quad I_7
 =\frac{1}{W^4}
       \left(
         \frac{\mathcal T'}{\mathcal T}z
         +\left\langle\nabla u,
                       \frac{\nabla\phi}{\phi}\right\rangle_g
       \right)^2,\\
	   I_8
 &\le\frac{m_1\sqrt z}{W}\log^\theta(1+z)
              \left|\frac{\nabla\phi}{\phi}\right|_g\le\varepsilon_3\frac{W^2}{z}\log^{2\theta}(1+z)
       +\frac{m_1^2}{4\varepsilon_3}
              \left|\frac{\nabla\phi}{\phi}\right|_g^2,\\
 I_9&=2(n-1)\kappa^2,\quad I_{10}
 =C_{n,\alpha}K(R_0,\mu)\phi^{-\frac{2}{\alpha}}.
\end{aligned}
\]
Expanding the squares in $I_6$ and $I_7$ gives
\[
\begin{aligned}
 I_5+I_6+I_7
 &=\left(\frac{\mathcal T'}{\mathcal T}\right)^2
       \frac{z}{W^2}
       \left(b+\frac{z-2}{2W^2}\right)\\
 &\quad+
 \frac{\mathcal T'}{\mathcal T}\frac{z-2}{W^4}
       \left\langle\nabla u,
                    \frac{\nabla\phi}{\phi}\right\rangle_g\\
 &\quad-\frac{2+z}{2W^4}
              \left|\frac{\nabla\phi}{\phi}\right|_g^2
       +\frac{1}{W^4}
        \left\langle\nabla u,
                     \frac{\nabla\phi}{\phi}\right\rangle_g^2.
\end{aligned}
\]
Using
\[
 \frac{z-2}{2W^2}\ge-1,\quad
 \frac{|z-2|}{W^3}\le2,\quad
 \frac{2+z}{2W^4}\le1,
\]
we estimate the mixed term by
\[
\begin{aligned}
 \left|
 \frac{\mathcal T'}{\mathcal T}\frac{z-2}{W^4}
       \left\langle\nabla u,
                    \frac{\nabla\phi}{\phi}\right\rangle_g
 \right|
 &\le
 2\left|\frac{\mathcal T'}{\mathcal T}\right|
       \frac{\sqrt z}{W}
       \left|\frac{\nabla\phi}{\phi}\right|_g\\
 &\le
 \varepsilon_4
       \left(\frac{\mathcal T'}{\mathcal T}\right)^2
       \frac{z}{W^2}
 +\frac{1}{\varepsilon_4}
       \left|\frac{\nabla\phi}{\phi}\right|_g^2.
\end{aligned}
\]
Hence
\[
 I_5+I_6+I_7
 \ge
 (b-1-\varepsilon_4)
       \left(\frac{\mathcal T'}{\mathcal T}\right)^2
       \frac{z}{W^2}
 -\left(1+\frac{1}{\varepsilon_4}\right)
       \left|\frac{\nabla\phi}{\phi}\right|_g^2.
\]

By \eqref{eq:log-elementary},
\[
 \frac{z}{W^2}\log^{2\theta-2}(1+z)
 \le\frac{W^2}{z}\log^{2\theta}(1+z).
\]
Combining the preceding estimates in \eqref{eq:master},
and using $W\ge1$, gives
\begin{equation}\label{eq-part3-1}
\begin{aligned}
 &2m_1\theta
       \left|\frac{\mathcal T'}{\mathcal T}\right|
       \frac{z}{W}\log^{\theta-1}(1+z)\\
 &\quad+
 \left(\frac{2m_1^2}{n}
       -\varepsilon_1-\varepsilon_2-\varepsilon_3\right)
       \frac{W^2}{z}\log^{2\theta}(1+z)\\
 &\quad+
 \left(b-1-\varepsilon_4
             -\frac{m_1^2}{4\varepsilon_2}\right)
       \left(\frac{\mathcal T'}{\mathcal T}\right)^2
       \frac{z}{W^2}\\
 &\le
 \left(
 1+\frac{1}{\varepsilon_4}
  +\frac{m_1^2\theta^2}{\varepsilon_1}
  +\frac{m_1^2}{4\varepsilon_3}
 \right)
       \left|\frac{\nabla\phi}{\phi}\right|_g^2
 +2(n-1)\kappa^2
 +C_{n,\alpha}K(R_0,\mu)\phi^{-\frac{2}{\alpha}}.
\end{aligned}
\end{equation}

We now choose
\[
 \varepsilon_1=\varepsilon_2=\varepsilon_3
      =\frac{m_1^2}{3n},\quad
 \varepsilon_4=1,\quad
 b=3+\frac{3n}{4}.
\]
Then
\[
 \frac{2m_1^2}{n}
       -\varepsilon_1-\varepsilon_2-\varepsilon_3
 =\frac{m_1^2}{n},\quad
 b-1-\varepsilon_4-\frac{m_1^2}{4\varepsilon_2}=1.
\]
Combining \eqref{eq-part3-1} and Lemma \ref{lem:cutoff} yields
\[
\begin{aligned}
 &2m_1\theta
       \left|\frac{\mathcal T'}{\mathcal T}\right|
       \frac{z}{W}\log^{\theta-1}(1+z)
 +\frac{m_1^2}{n}\frac{W^2}{z}\log^{2\theta}(1+z)+
       \left(\frac{\mathcal T'}{\mathcal T}\right)^2
       \frac{z}{W^2}\\
 &\quad \le C\Lambda(R_0,\mu,\kappa,0)
                       \phi^{-\frac{2}{\alpha}}.
\end{aligned}
\]
Since
\[
 \left|\frac{\mathcal T'}{\mathcal T}\right|
 \ge\frac{1}{2b(L+1)},
\]
we conclude that
\[
 \frac{z}{W}\log^{\theta-1}(1+z)
 \le C(L+1)\Lambda(R_0,\mu,\kappa,0)
                       \phi^{-\frac{2}{\alpha}}.
\]

If $0<z(q)\le1$, then $\theta\le1$ and
\eqref{eq:log-elementary} imply
\[
 \frac{z}{W}\log^{\theta-1}(1+z)
 \ge\frac{1}{\sqrt2}z^\theta.
\]
Hence
\[
\begin{aligned}
 z(q)\phi(q)
 &\le C(L+1)^{\frac{1}{\theta}}
       \left[\Lambda(R_0,\mu,\kappa,0)\right]^{
                         \frac{1}{\theta}}
       \phi(q)^{1-\frac{2}{\alpha\theta}}\\
 &\le C(L+1)^{\frac{1}{\theta}}
       \left[\Lambda(R_0,\mu,\kappa,0)\right]^{
                         \frac{1}{\theta}}.
\end{aligned}
\]
If $z(q)\ge1$, the elementary logarithmic growth bound
\[
 \log^{1-\theta}(1+z)
 \le C_{\theta,\beta}z^{\frac{\beta}{2}}
\]
and $W\le\sqrt{2z}$ give
\[
 \frac{z}{W}\log^{\theta-1}(1+z)
 \ge C_{\theta,\beta}^{-1}
                    z^{\frac{1-\beta}{2}}.
\]
It follows that
\[
\begin{aligned}
 z(q)\phi(q)
 &\le C(L+1)^{\frac{2}{1-\beta}}
       \left[\Lambda(R_0,\mu,\kappa,0)\right]^{
                         \frac{2}{1-\beta}}
       \phi(q)^{1-\frac{4}{\alpha(1-\beta)}}\\
 &\le C(L+1)^{\frac{2}{1-\beta}}
       \left[\Lambda(R_0,\mu,\kappa,0)\right]^{
                         \frac{2}{1-\beta}}.
\end{aligned}
\]
Both final inequalities follow from our choice of $\alpha$.

Finally, maximality and the bounds for $\mathcal T$ imply
\[
 z(p)
 \le\frac{\mathcal T(u(q))}{\mathcal T(u(p))}
              z(q)\phi(q)
 \le2^{\frac{1}{b}}z(q)\phi(q).
\]
Combining the two cases and taking square roots yields
\[
\begin{aligned}
 |\nabla u(p)|_g\le\max\Big\{
 &C_1(L+1)^{\frac{1}{2\theta}}
       \left[\Lambda(R_0,\mu,\kappa,0)\right]^{
                          \frac{1}{2\theta}},\ \ C_2(L+1)^{\frac{1}{1-\beta}}
       \left[\Lambda(R_0,\mu,\kappa,0)\right]^{
                          \frac{1}{1-\beta}}
 \Big\}.
\end{aligned}
\]
Here both constants depend only on $n,m_1,\theta,\beta$.
This proves \eqref{eq:c-grad} at radius $R_0$.

\subsection{Proof of part \textup{(iv)}}

Condition \eqref{A4} controls the vertical derivative
by $f^2$, while the dependence of $\mathcal T$ on $u$
provides the remaining positive term.
Choose
\[
 \mathcal T(u)
 =\left(L+1+u-\inf\nolimits_{B_R(p)}u\right)^{-\frac{1}{b}},\quad F(z)=\log(1+z),\quad
 \alpha=2,
\]
where $b>0$ will be fixed below.
Then
\[
 [2(L+1)]^{-\frac{1}{b}}
 \le\mathcal T\le(L+1)^{-\frac{1}{b}},
\]
and
\[
 \frac{\mathcal T'}{\mathcal T}
 =-\frac{1}{b\left(L+1+u-\inf\nolimits_{B_R(p)}u\right)},\quad
 \frac{\mathcal T''}{\mathcal T}
   -\left(\frac{\mathcal T'}{\mathcal T}\right)^2
 =b\left(\frac{\mathcal T'}{\mathcal T}\right)^2.
\]
At $q$, \eqref{eq-gradz-general} reads
\[
 \nabla z
 =-W^2\log(1+z)
    \left(
      \frac{\mathcal T'}{\mathcal T}\nabla u
      +\frac{\nabla\phi}{\phi}
    \right).
\]
Let
$\varepsilon_1,\varepsilon_2,\varepsilon_3,\varepsilon_4>0$.
Using \eqref{eq-f}, \eqref{A4}, 
\eqref{eq:log-elementary}, \eqref{eq-b-PQlog}, and
Young's inequality, we obtain
\[
\begin{aligned}
 I_1
 &\ge-2\varepsilon_1\frac{f^2}{\log(1+z)}
       -\frac{\sigma_h^2}{2\varepsilon_1},\\
 I_2
 &=-W\left\langle
        \nabla^v f,
        \frac{\mathcal T'}{\mathcal T}\nabla u
        +\frac{\nabla\phi}{\phi}
       \right\rangle_g\\
 &\ge-\varepsilon_2W^4|\nabla^v f|_g^2
       -\frac{1}{4\varepsilon_2W^2}
        \left|
          \frac{\mathcal T'}{\mathcal T}\nabla u
          +\frac{\nabla\phi}{\phi}
        \right|_g^2\\
 &\ge-\frac{\varepsilon_2}{m_1}
                    \frac{f^2}{\log(1+z)}
       -\frac{1}{2\varepsilon_2}
          \left(\frac{\mathcal T'}{\mathcal T}\right)^2
          \frac{z}{W^2}
       -\frac{1}{2\varepsilon_2}
          \left|\frac{\nabla\phi}{\phi}\right|_g^2,\\
 I_3
 &=\frac{2}{n}\frac{f^2}{\log(1+z)},\\
 I_4
 &=\frac{\mathcal T'}{\mathcal T}\frac{f}{W}\ge-\varepsilon_3\frac{f^2}{\log(1+z)}
       -\frac{1}{4\varepsilon_3}
          \left(\frac{\mathcal T'}{\mathcal T}\right)^2
          \frac{\log(1+z)}{W^2}\\
 &\ge-\varepsilon_3\frac{f^2}{\log(1+z)}
       -\frac{1}{4\varepsilon_3}
          \left(\frac{\mathcal T'}{\mathcal T}\right)^2
          \frac{z}{W^2},\\
 I_5
 &=b\left(\frac{\mathcal T'}{\mathcal T}\right)^2
                      \frac{z}{W^2},\\ I_6
 &=\left(-1+\frac{\log(1+z)}{2W^2}\right)
       \left|
         \frac{\mathcal T'}{\mathcal T}\nabla u
         +\frac{\nabla\phi}{\phi}
       \right|_g^2,\\
I_7
 &=\frac{1}{W^2}
       \left(
         \frac{\mathcal T'}{\mathcal T}z
         +\left\langle\nabla u,
                       \frac{\nabla\phi}{\phi}\right\rangle_g
       \right)^2,\\
	    I_8
 &\le m_2\sqrt{\frac{z}{W^2}}
              \left|\frac{\nabla\phi}{\phi}\right|_g,\\
 I_9
 &=2(n-1)\kappa^2\frac{z}{W^2\log(1+z)}
 \le2(n-1)\kappa^2,\ \ I_{10}
 =C_{n,2}K(R_0,\mu)\phi^{-1}.
\end{aligned}
\]
Here the first inequality in \eqref{A4} is used
in the last estimate of $I_2$, while the bound
$|f|\le m_2$ is used to estimate $I_8$.

Expanding the squares in $I_6$ and $I_7$ gives
\[
\begin{aligned}
 I_5+I_6+I_7
 &=\left(\frac{\mathcal T'}{\mathcal T}\right)^2
       \frac{z}{W^2}
       \left(b-1+\frac{\log(1+z)}{2}\right)\\
 &\quad+
 \frac{\mathcal T'}{\mathcal T}
       \frac{\log(1+z)-2}{W^2}
       \left\langle\nabla u,
                    \frac{\nabla\phi}{\phi}\right\rangle_g\\
 &\quad+
 \left(-1+\frac{\log(1+z)}{2W^2}\right)
              \left|\frac{\nabla\phi}{\phi}\right|_g^2
 +\frac{1}{W^2}
       \left\langle\nabla u,
                    \frac{\nabla\phi}{\phi}\right\rangle_g^2.
\end{aligned}
\]
Since $|\log(1+z)-2|\le2W$, Young's inequality gives
\[
\begin{aligned}
 \left|
 \frac{\mathcal T'}{\mathcal T}
       \frac{\log(1+z)-2}{W^2}
       \left\langle\nabla u,
                    \frac{\nabla\phi}{\phi}\right\rangle_g
 \right|
 \le
 2\left|\frac{\mathcal T'}{\mathcal T}\right|
       \frac{\sqrt z}{W}
       \left|\frac{\nabla\phi}{\phi}\right|_g\le
 \varepsilon_4
       \left(\frac{\mathcal T'}{\mathcal T}\right)^2
       \frac{z}{W^2}
 +\frac{1}{\varepsilon_4}
       \left|\frac{\nabla\phi}{\phi}\right|_g^2.
\end{aligned}
\]
Consequently,
\begin{equation}\label{eq-d-I5672}
\begin{aligned}
 I_5+I_6+I_7
 &\ge
 \left(b-1-\varepsilon_4+\frac{\log(1+z)}{2}\right)
       \left(\frac{\mathcal T'}{\mathcal T}\right)^2
       \frac{z}{W^2}-
 \left(1+\frac{1}{\varepsilon_4}\right)
       \left|\frac{\nabla\phi}{\phi}\right|_g^2.
\end{aligned}
\end{equation}

Substituting these estimates into \eqref{eq:master} yields
\begin{equation}\label{eq-part4-1}
\begin{aligned}
 &\left(
   \frac{2}{n}-2\varepsilon_1
   -\frac{\varepsilon_2}{m_1}-\varepsilon_3
  \right)\frac{f^2}{\log(1+z)}\\
 &\quad+
 \left(
   b-1-\varepsilon_4
   -\frac{1}{2\varepsilon_2}
   -\frac{1}{4\varepsilon_3}
   +\frac{\log(1+z)}{2}
 \right)
       \left(\frac{\mathcal T'}{\mathcal T}\right)^2
       \frac{z}{W^2}\\
 &\le
 m_2\sqrt{\frac{z}{W^2}}
              \left|\frac{\nabla\phi}{\phi}\right|_g
 +\left(
   1+\frac{1}{\varepsilon_4}
    +\frac{1}{2\varepsilon_2}
  \right)
              \left|\frac{\nabla\phi}{\phi}\right|_g^2\\
 &\quad+
 \frac{\sigma_h^2}{2\varepsilon_1}
 +2(n-1)\kappa^2
 +C_{n,2}K(R_0,\mu)\phi^{-1}.
\end{aligned}
\end{equation}
We now choose
\[
 \varepsilon_1=\frac{1}{4n},\quad
 \varepsilon_2=\frac{m_1}{4n},\quad
 \varepsilon_3=\frac{1}{4n},\quad
 \varepsilon_4=1,\quad
 b=3+\frac{2n}{m_1}+n.
\]
Then
\[
\begin{aligned}
 \frac{2}{n}-2\varepsilon_1
       -\frac{\varepsilon_2}{m_1}-\varepsilon_3
 =\frac{1}{n},\quad 
 b-1-\varepsilon_4
       -\frac{1}{2\varepsilon_2}
       -\frac{1}{4\varepsilon_3}
 &=1.
\end{aligned}
\]
Using \eqref{eq-part4-1} and Lemma \ref{lem:cutoff},
we conclude that
\[
\begin{aligned}
 &\frac{1}{n}\frac{f^2}{\log(1+z)}
 +\left(1+\frac{\log(1+z)}{2}\right)
       \left(\frac{\mathcal T'}{\mathcal T}\right)^2
       \frac{z}{W^2}\\
 &\quad\le
 C\Big(
 \sqrt{\frac{z}{W^2}}
       \left[\Lambda(R_0,\mu,\kappa,\sigma_h)\right]^{
                         \frac{1}{2}}
       \phi^{-\frac{1}{2}}
 +\Lambda(R_0,\mu,\kappa,\sigma_h)\phi^{-1}
 \Big).
\end{aligned}
\]

Since
\[
 \left|\frac{\mathcal T'}{\mathcal T}\right|^{-1}
 =b\left(L+1+u-\inf\nolimits_{B_R(p)}u\right)
 \le2b(L+1),
\]
inequality \eqref{eq:log-elementary} and
$\frac{z}{W^2}\le1$ give
\[
\begin{aligned}
 \log(1+z)
 \le C(L+1)^2\Big(
 &[\Lambda(R_0,\mu,\kappa,\sigma_h)]^{\frac{1}{2}}
       \phi^{-\frac{1}{2}}+\Lambda(R_0,\mu,\kappa,\sigma_h)\phi^{-1}
 \Big).
\end{aligned}
\]
Multiplying by $\phi(q)$ and using $0<\phi(q)\le1$,
we obtain
\[
\begin{aligned}
 \log(1+z(q))\phi(q)
 &\le C(L+1)^2\Big(
 [\Lambda(R_0,\mu,\kappa,\sigma_h)]^{\frac{1}{2}}
 +\Lambda(R_0,\mu,\kappa,\sigma_h)
 \Big)\\
 &\le C(L+1)^2
 \max\left\{
 \Lambda(R_0,\mu,\kappa,\sigma_h),\ \
 [\Lambda(R_0,\mu,\kappa,\sigma_h)]^{\frac{1}{2}}
 \right\}.
\end{aligned}
\]
The maximality of $\mathcal T(u)\log(1+z)\phi$,
together with $\phi(p)=1$ and the bounds for
$\mathcal T$, yields
\[
\begin{aligned}
 \log(1+z(p))
 &\le
 \frac{\mathcal T(u(q))}{\mathcal T(u(p))}
       \log(1+z(q))\phi(q)\\
 &\le2^{\frac{1}{b}}\log(1+z(q))\phi(q)\\
 &\le C(L+1)^2
 \max\left\{
 \Lambda(R_0,\mu,\kappa,\sigma_h),
 [\Lambda(R_0,\mu,\kappa,\sigma_h)]^{\frac{1}{2}}
 \right\}.
\end{aligned}
\]
Here $C=C(n,m_1,m_2)$.
Exponentiating proves \eqref{eq:d-grad}
at radius $R_0$.

\subsection{Proof of part \textup{(v)}}

When $f\equiv0$, the terms involving $f$ and its
derivatives vanish. 
Choose
\[
 \mathcal T(u)
 =\left(L+u-\inf\nolimits_{B_R(p)}u\right)^{-\frac{1}{b}},\quad F(z)=\log(1+z),\quad \alpha=2,
\]
where $b>0$ will be fixed below. Then
\[
 (2L)^{-\frac{1}{b}}\le\mathcal T\le L^{-\frac{1}{b}},
\]
and
\[
 \frac{\mathcal T'}{\mathcal T}
 =-\frac{1}{b\left(L+u-\inf\nolimits_{B_R(p)}u\right)},\quad
 \frac{\mathcal T''}{\mathcal T}
   -\left(\frac{\mathcal T'}{\mathcal T}\right)^2
 =b\left(\frac{\mathcal T'}{\mathcal T}\right)^2.
\]

At $q$, \eqref{eq-gradz-general} gives
\[
 \nabla z
 =-W^2\log(1+z)
    \left(
      \frac{\mathcal T'}{\mathcal T}\nabla u
      +\frac{\nabla\phi}{\phi}
    \right).
\]
Using $f\equiv0$ and \eqref{eq-b-PQlog}, the ten terms become
\[
\begin{aligned}
 I_1&=0,\quad I_2=0,\quad I_3=0,\quad I_4=0,\\
 I_5
 &=b\left(\frac{\mathcal T'}{\mathcal T}\right)^2
                     \frac{z}{W^2},\\
 I_6
 &=\left(-1+\frac{\log(1+z)}{2W^2}\right)
       \left|
         \frac{\mathcal T'}{\mathcal T}\nabla u
         +\frac{\nabla\phi}{\phi}
       \right|_g^2,\\
 I_7
 &=\frac{1}{W^2}
       \left(
         \frac{\mathcal T'}{\mathcal T}z
         +\left\langle\nabla u,
                       \frac{\nabla\phi}{\phi}\right\rangle_g
       \right)^2,\\
 I_8&=0,\\
 I_9
 &=2(n-1)\kappa^2\frac{z}{W^2\log(1+z)}
 \le2(n-1)\kappa^2,\\
 I_{10}
 &=C_{n,2}K(R_0,\mu)\phi^{-1}.
\end{aligned}
\]
For every $\varepsilon_4>0$, the calculation leading to
\eqref{eq-d-I5672} remains valid. Hence
\eqref{eq:master} gives
\begin{equation}\label{eq-part5-1}
\begin{aligned}
 &\left(b-1-\varepsilon_4+\frac{\log(1+z)}{2}\right)
       \left(\frac{\mathcal T'}{\mathcal T}\right)^2
       \frac{z}{W^2}\\
 &\quad\le
 \left(1+\frac{1}{\varepsilon_4}\right)
       \left|\frac{\nabla\phi}{\phi}\right|_g^2
 +2(n-1)\kappa^2
 +C_{n,2}K(R_0,\mu)\phi^{-1}.
\end{aligned}
\end{equation}
Choose
\[
 \varepsilon_4=1,\quad b=3.
\]
Inequality~\eqref{eq-part5-1} and Lemma \ref{lem:cutoff} then imply
\[
 \left(1+\frac{\log(1+z)}{2}\right)
       \left(\frac{\mathcal T'}{\mathcal T}\right)^2
       \frac{z}{W^2}
 \le C_n\Lambda(R_0,\mu,\kappa,0)\phi^{-1}.
\]
Since $b=3$, we have
\[
 \left|\frac{\mathcal T'}{\mathcal T}\right|^{-1}
 =3\left(L+u-\inf\nolimits_{B_R(p)}u\right)
 \le6L.
\]
It follows from \eqref{eq:log-elementary} that
\[
 \log(1+z)
 \le C_nL^2\Lambda(R_0,\mu,\kappa,0)\phi^{-1}.
\]
Multiplying by $\phi(q)$ gives
\[
 \log(1+z(q))\phi(q)
 \le C_nL^2\Lambda(R_0,\mu,\kappa,0).
\]
By maximality and the bounds for $\mathcal T$,
\[
\begin{aligned}
 \log(1+z(p))
 &\le
 \frac{\mathcal T(u(q))}{\mathcal T(u(p))}
       \log(1+z(q))\phi(q)\\
 &\le2^{\frac{1}{3}}\log(1+z(q))\phi(q)\\
 &\le C_nL^2\Lambda(R_0,\mu,\kappa,0).
\end{aligned}
\]
Exponentiating proves \eqref{eq:e-grad}
at radius $R_0$.



The passage to the original radius is common to all five parts. 
The constants are independent of $R_0$,
while $R$ and $L=\osc_{B_R(p)}u$ remain fixed
as $R_0\uparrow R$.
Since
\[
 K(R_0,\mu)\longrightarrow K(R,\mu),\quad
 \Lambda(R_0,\mu,\kappa,\sigma_h)
 \longrightarrow\Lambda(R,\mu,\kappa,\sigma_h)
 \quad\text{as }R_0\rightarrow R,
\]
letting $R_0\rightarrow R$ gives all the asserted estimates
and completes the proof of Theorem~\ref{thm-main}.

\section{Proofs of Corollaries~\ref{cor1} and~\ref{cor2}} 
\label{s4}
\begin{proof}[Proof of Corollary~\ref{cor1}]
We verify the Ricci lower bound and the distance
Hessian comparison required in Theorem~\ref{thm-main}.

Fix $q\in B_R(p)$ and a unit vector $X\in T_qM$.
Choose a local orthonormal frame normal at $q$ with
$e_1=X$. By the curvature convention of
Section~\ref{s21},
\[
 \Ric_g(X,X)
 =\sum_{i=2}^n\mathcal R(X,e_i,e_i,X)
 =\sum_{i=2}^n\Sec_g([X\wedge e_i])
 \ge-(n-1)\rho^2.
\]
Hence
\[
 \Ric_g\ge-(n-1)\rho^2g
 \quad\text{on }B_R(p).
\]

For
$q\in B_R(p)\setminus(\{p\}\cup\Cut(p))$,
the smooth Hessian comparison
\cite[Theorem~6.4.3]{Petersen2016} gives
\[
 \Hess r
 \le c_\rho(r(q))
       \bigl(g-\dd r\otimes\dd r\bigr)
 \quad\text{at }q.
\]
Thus \eqref{eq-comp2} holds with $\mu=\rho$.

Now let $q\in B_R(p)\cap\Cut(p)$ and fix
$\varepsilon>0$. Choose a unit-speed minimizing
geodesic
\[
 \gamma:[0,r(q)]\longrightarrow M,\quad
 \gamma(0)=p,\quad \gamma(r(q))=q.
\]
By continuity of $c_\rho$ on $(0,\infty)$, we may
choose $0<\delta<r(q)$ such that
\[
 c_\rho(r(q)-\delta)
 \le c_\rho(r(q))+\varepsilon.
\]
The smoothness argument in the proof of
\cite[Lemma~7.1.9]{Petersen2016} shows that
\[
 q\notin\Cut(\gamma(\delta)).
\]
Consequently,
\[
 r_{q,\varepsilon}(x)
 =\delta+d_g(\gamma(\delta),x)
\]
is smooth on a sufficiently small neighborhood of $q$.
The triangle inequality and the minimizing property
of $\gamma$ give
\[
 r_{q,\varepsilon}\ge r,\quad
 r_{q,\varepsilon}(q)=r(q),\quad
 |\nabla r_{q,\varepsilon}(q)|_g=1.
\]
Thus $r_{q,\varepsilon}$ is a smooth upper support
for $r$ at $q$.

Moreover,
\[
 q\in B_{R-\delta}(\gamma(\delta))
 \subset B_R(p),
\]
so the sectional curvature bound is available along
the minimizing segment from $\gamma(\delta)$ to $q$.
Applying the smooth Hessian comparison to
$d_g(\gamma(\delta),\cdot)$ gives, at $q$,
\[
 \Hess r_{q,\varepsilon}
 \le c_\rho(r(q)-\delta)
       \bigl(
         g-\dd r_{q,\varepsilon}
              \otimes\dd r_{q,\varepsilon}
       \bigr).
\]
For every $X\in T_qM$,
\[
\begin{aligned}
 &\bigl(
   g-\dd r_{q,\varepsilon}\otimes\dd r_{q,\varepsilon}
  \bigr)(X,X)=
 |X|_g^2
 -\langle\nabla r_{q,\varepsilon}(q),X\rangle_g^2
 \ge0.
\end{aligned}
\]
Our choice of $\delta$ therefore yields
\[
 \Hess r_{q,\varepsilon}
 \le\bigl(c_\rho(r(q))+\varepsilon\bigr)
       \bigl(
         g-\dd r_{q,\varepsilon}
              \otimes\dd r_{q,\varepsilon}
       \bigr)
 \quad\text{at }q.
\]
This implies \eqref{eq-comp3}, and hence
Definition~\ref{def-compare} holds with $\mu=\rho$.

Both geometric hypotheses of Theorem~\ref{thm-main}
are now verified with $\kappa=\mu=\rho$.
Substituting these values into \eqref{eq-scales} gives
\[
 \Lambda(R,\rho,\rho,\sigma_h)
 =\frac{1+\rho R}{R^2}+(n-1)\rho^2+n\sigma_h^2.
\]
The corollary follows from Theorem~\ref{thm-main}.
\end{proof}

\begin{proof}[Proof of Corollary~\ref{cor2}]
The preceding proof supplies the distance Hessian
comparison with $\mu=\rho$ on every geodesic ball.
Since $\Ric_g\ge0$ and \eqref{eq-f} holds with
$\sigma_h=0$, Theorem~\ref{thm-main} applies with
\[
 \kappa=0,\quad \mu=\rho,\quad \sigma_h=0.
\]

Fix an arbitrary $q\in M$. For $R>2d_g(p,q)$,
\[
 B_{\frac{R}{2}}(q)\subset B_R(p),\quad
 \osc_{B_{\frac{R}{2}}(q)}u\le L.
\]
Completeness and continuity of $u$ ensure that these
oscillations are finite for every finite $R$.
We apply Theorem~\ref{thm-main} on
$B_{\frac{R}{2}}(q)$ and use $L$ to bound its
oscillation. The geometric factor is
\[
 \Lambda\left(\frac{R}{2},\rho,0,0\right)
 =K\left(\frac{R}{2},\rho\right)
 =\frac{4}{R^2}+\frac{2\rho}{R}
 \longrightarrow0.
\]
In particular, for $R\ge1$,
\[
 \Lambda\left(\frac{R}{2},\rho,0,0\right)
 \le
 \begin{cases}
 \frac{4+2\rho}{R},&\rho>0,\\
 \frac{4}{R^2},&\rho=0.
 \end{cases}
\]
All constants in the gradient estimates below are
independent of $R$.

\smallskip
\noindent\textup{(i)}
Under \eqref{A1}, \eqref{eq:a-grad}
gives
\[
 |\nabla u(q)|_g
 \le C
 \left[
   \Lambda\left(\frac{R}{2},\rho,0,0\right)
 \right]^{\frac{1}{2\theta}}
 \longrightarrow0.
\]

\smallskip
\noindent\textup{(ii)}
Under \eqref{A2}, for all sufficiently large $R$,
\[
 \Lambda\left(\frac{R}{2},\rho,0,0\right)\le1.
\]
Since $\theta>1$, estimate \eqref{eq:b-grad} and
$e^{Ct}-1\le Ce^Ct$ for $0\le t\le1$ give
\[
 |\nabla u(q)|_g^2
 \le C
 \left[
  \Lambda\left(\frac{R}{2},\rho,0,0\right)
 \right]^{\frac{1}{2\theta-1}}
 \longrightarrow0.
\]

\smallskip
\noindent\textup{(iii)}
The oscillation assumption implies
\[
 (L+1)\Lambda\left(\frac{R}{2},\rho,0,0\right)
 \le
 \begin{cases}
 (4+2\rho)\frac{L+1}{R},&\rho>0,\\
 4\frac{L+1}{R^2},&\rho=0,
 \end{cases}
 \longrightarrow0.
\]
Taking $\beta=\frac{1}{2}$ in \eqref{eq:c-grad},
we obtain
\[
\begin{aligned}
 |\nabla u(q)|_g\le\max\Bigg\{
 &C_1
 \left[
  (L+1)\Lambda\left(\frac{R}{2},\rho,0,0\right)
 \right]^{\frac{1}{2\theta}},\ \ C_2
 \left[
  (L+1)\Lambda\left(\frac{R}{2},\rho,0,0\right)
 \right]^2
 \Bigg\}
 \longrightarrow0.
\end{aligned}
\]

\smallskip
\noindent\textup{(iv)}
For all sufficiently large $R$,
$\Lambda(\frac{R}{2},\rho,0,0)\le1$.
Taking logarithms in \eqref{eq:d-grad} therefore gives
\[
 \log(1+|\nabla u(q)|_g^2)
 \le C(L+1)^2
 \left[
  \Lambda\left(\frac{R}{2},\rho,0,0\right)
 \right]^{\frac{1}{2}}.
\]
The right-hand side tends to zero, since
\[
\begin{aligned}
 &(L+1)^2
 \left[
  \Lambda\left(\frac{R}{2},\rho,0,0\right)
 \right]^{\frac{1}{2}}\le
 \begin{cases}
 (4+2\rho)^{\frac{1}{2}}
 \left(\frac{L+1}{R^{\frac{1}{4}}}\right)^2,
       &\rho>0,\\
 4^{\frac{1}{2}}
 \left(\frac{L+1}{R^{\frac{1}{2}}}\right)^2,
       &\rho=0,
 \end{cases}
 \longrightarrow0.
\end{aligned}
\]

\smallskip
\noindent\textup{(v)}
 Taking logarithms in \eqref{eq:e-grad}, we conclude that
\[
\begin{aligned}
\log(1+|\nabla u(q)|_g^2)
&\le CL^2
\Lambda\left(\frac R2,\rho,0,0\right)\\
&\le C
\begin{cases}
(4+2\rho)\left(\frac L{R^{\frac12}}\right)^2,
&\rho>0,\\[4pt]
4\left(\frac LR\right)^2,
&\rho=0,
\end{cases}\longrightarrow0.
\end{aligned}
\]


In each case, letting $R\to\infty$ gives
$\nabla u(q)=0$. Since $q\in M$ is arbitrary,
$\nabla u\equiv0$ on $M$. 
\end{proof}




\end{document}